\documentclass[11pt, reqno]{amsart}
\usepackage{amsmath}
\usepackage{amsthm}
\usepackage{amsfonts}
\usepackage{amssymb}
\usepackage{mathtools}
\usepackage{stackengine}
\usepackage[mathscr]{euscript}
\usepackage{bm}
\usepackage{hyperref}
\usepackage[table]{xcolor}
\usepackage[all,cmtip]{xy}
\usepackage{enumerate}
\usepackage{enumitem}
\usepackage{tikz}  
\usepackage{tikz-cd}
\usetikzlibrary{shapes}
\tikzstyle{dot}=[circle,draw,fill=black,inner sep=0pt, minimum width=2pt]
\usepackage{graphicx}
\usepackage{scalerel}
\usepackage{microtype}
\usepackage{verbatim}
\usepackage{libertine}
\usepackage[a4paper,top=3cm,bottom=2.5cm,left=3cm,right=3cm,marginparwidth=1.75cm]{geometry}
\usepackage{pgfplots}
\pgfplotsset{compat=1.15}
\usepackage{mathrsfs}
\usetikzlibrary{shapes.geometric, calc, arrows}
\usepackage[font=small,labelfont=bf]{caption}
\definecolor{xdxdff}{rgb}{0.49019607843137253,0.49019607843137253,1.}
\definecolor{uuuuuu}{rgb}{0.26666666666666666,0.26666666666666666,0.26666666666666666}
\definecolor{ududff}{rgb}{0.30196078431372547,0.30196078431372547,1.}
\definecolor{cite}{RGB}{44,123,182}
\definecolor{ref}{RGB}{215,25,28}
\theoremstyle{plain}
\newtheorem{thm}{Theorem}[section]
\newtheorem{cor}[thm]{Corollary}
\newtheorem{corol}[thm]{Corollary}

\newtheorem{lem}[thm]{Lemma}

\newtheorem{prop}[thm]{Proposition}

\newtheorem*{thm*}{Theorem}
\newtheorem*{corollary*}{Corollary}
\newtheorem*{lemma*}{Lemma}
\newtheorem*{ld*}{Lemma/Definition}
\newtheorem*{proposition*}{Proposition}

\theoremstyle{definition}
\newtheorem{dfn}[thm]{Definition}
\newtheorem{rem}[thm]{Remark}
\newtheorem{example}[thm]{Example}

\newtheorem*{def*}{Definition}
\newtheorem*{rem*}{Remark}
\newtheorem*{example*}{Example}
\newtheorem*{xca*}{Exercise}
\newtheorem*{claim*}{Claim}
\newtheorem*{fact*}{Fact}
\newtheorem*{notation*}{Notation}
\newtheorem*{construction*}{Construction}
\newtheorem*{ack*}{Acknowledgements}
\newtheorem*{que*}{Question}
\newtheorem*{problem*}{Problem}
\newtheorem*{con*}{Conjecture}
\newtheorem*{assumption*}{Assumption}
\newcommand{\PP}{\mathbb{P}}
\newcommand{\NN}{\mathbb{N}}
\newcommand{\CC}{\mathbb{C}}
\newcommand{\OO}{\mathcal{O}}
\newcommand{\Z}{\mathbb{Z}}
\newcommand{\ZZ}{\mathbb{Z}}
\newcommand{\QQ}{\mathbb{Q}}
\newcommand{\RR}{\mathbb{R}}

\newcommand{\HH}{\mathbb{H}}
\newcommand{\Ps}{\mathbb{P}}

\newcommand{\TT}{\mathbb{T}}
\newcommand{\ga}{\gamma}
\newcommand{\be}{\beta}
\newcommand{\Gm}{{\mathbb{C}^\times}}

\docsvlist{A,B,C,D,E,F,G,H,I,J,K,L,M,N,O,P,Q,R,S,T,U,V,W,X,Y,Z} 

\docsvlist{A,B,C,D,E,F,G,H,I,J,K,L,M,N,O,P,Q,R,S,T,U,V,W,X,Y,Z} 

\docsvlist{A,B,C,D,E,F,G,H,I,J,K,L,M,N,O,P,Q,R,S,T,U,V,W,X,Y,Z} 

\docsvlist{A,B,C,D,E,F,G,H,I,J,K,L,M,N,O,P,Q,R,S,T,U,V,W,X,Y,Z} 

\docsvlist{A,B,C,D,E,F,G,H,I,J,K,L,M,N,O,P,Q,R,S,T,U,V,W,X,Y,Z,
a,b,c,d,e,f,g,h,i,j,k,l,m,n,o,p,q,r,s,t,u,v,w,x,y,z} 

\DeclareMathOperator{\rk}{rk}

\DeclareMathOperator{\GL}{GL}

\DeclareMathOperator{\Spec}{Spec}

\DeclareMathOperator{\gr}{gr}

\newcommand{\bs}{}

\begin{document}
\title{Geometric realisation of hypergeometric local systems}

\author[A. Abdelraouf]{Asem Abdelraouf}
\address{AA: Kyiv School of Economics, Mykoly Shpaka St, 3, Kyiv, Ukraine, 02000} 
\email{aabdelraouf@kse.org.ua}

\author[G. Gugiatti]{Giulia Gugiatti}
\address{GG: School of Mathematics, University of Edinburgh, Edinburgh EH9 3FD, United Kingdom} 
\email{giulia.gugiatti@ed.ac.uk}


\begin{abstract} 
We prove a realisation theorem for irreducible hypergeometric local systems defined over the rational numbers in terms of families of affine varieties in algebraic tori.
The families we consider have been studied extensively in the literature and appear in mirror symmetry. 
Our result holds unconditionally for families with one-dimensional or even-dimensional fibres. It holds under a monodromy assumption for families with fibres of odd dimension greater than one.
\end{abstract}
\maketitle

\setcounter{tocdepth}{1}
\tableofcontents
\section{Introduction}
Let $n$ be a positive integer, and let $\alpha=(\alpha_1,\dots,\alpha_n)$ and $\beta=(\beta_1,\dots,\beta_n)$ be two multisets of rational numbers. The hypergeometric differential operator associated to $\alpha,\beta$ is 
\begin{equation}
H=H(\alpha,\beta) \coloneqq \prod_{i=1}^n (\theta+\beta_i -1) - t \prod_{i=1}^n (\theta+\alpha_i)
\end{equation}
where $t$ is a coordinate on $\PP^1$ and $\theta:= t\frac{d}{dt}$. It is a differential operator on $\mathbb{P}^1$ with regular singularities at $0$, $1$, and $\infty$. The local system $\HH$ of solutions of $H \varphi=0$ is a complex local system on $U \coloneqq \mathbb{P}^{1}(\mathbb{C}) \setminus\{0,1, \infty\}$. We call such a local system a \emph{hypergeometric local system}. 

In this paper, we consider the subclass of irreducible hypergeometric local systems that are defined over $\QQ$.
In this case, $\alpha_i - \beta_j \notin \Z$ for all $i, j$. Moreover, the isomorphism class of $\HH$ can be given by a vector of nonzero integers $\gamma \in \Z^{l}$ with  $\gcd(\ga_1,\dots, \ga_l)=1$ and whose sum of entries is zero. 
The relationship between $\gamma$ and the parameters $\alpha, \beta$ is given by 

\begin{equation}
\label{eq:gamma-alpha-beta}
\frac{\prod_{\ga_j <0} \left( x^{-\ga_j} -1 \right) }{\prod_{\ga_j >0} \left( x^{\ga_j} -1 \right)} = \frac{ \prod_{j=1}^{n} (x- e^{2 \pi i \alpha_j} ) }{ \prod_{j=1}^{n} (x- e^{2 \pi i \beta_j} ) }
\end{equation}
\vspace{0.05 cm}

\noindent where irreducibility forces the numerator and denominator of the right-hand side of Equation \eqref{eq:gamma-alpha-beta} to have no common factors.

Associated to a vector $\gamma$ as above is a one-parameter family of affine varieties in $(\CC^\times)^{l-1}$, as follows.
Let $z_1, \dots, z_l$  be homogeneous coordinates on $(\CC^\times)^{l-1}$, and let $t$ be a coordinate on $ \CC^\times$.
Consider the affine variety $Z \subset (\CC^\times)^{l-1} \times \CC^\times$ defined by the homogeneous equations
\begin{equation}
     \begin{split}
    z_1+z_2+\dots +z_{l} &= 0 \\
    z_1^{\gamma_1}z_2^{\gamma_2} \cdots z_{l}^{\gamma_{l}}  &= \Gamma t
\end{split}
\end{equation}
where $ \Gamma\coloneqq \prod_{j=1}^l \gamma_j^{\gamma_j}$ and let $\pi \colon Z \to \CC^\times$ be the projection onto $t$. 
The fibres $Z_t \coloneqq \pi^{-1}(t)$ of the family $\pi$ are smooth unless $t=1$. We denote their dimension by $\kappa=l-3$.

The local system $\HH$ is expected to be realisable in terms of the family $(Z, \pi)$. More precisely, let $\pi_U$ be the restriction of $\pi$ to its regular locus $U$ and let $PR^\kappa \pi_{U !} \QQ$ be the sub-local system of the local system $R^\kappa \pi_{U !} \QQ$ consisting of primitive cohomology classes (see Section \ref{sec:local-systems} for details). Then, one expects an isomorphism of local systems over $U$:
\begin{equation}
\label{eq:H-(Z, pi)}
\HH    \simeq  \gr_\kappa^W PR^\kappa \pi_{U !}  \CC
\end{equation}
We expand on why \eqref{eq:H-(Z, pi)} is expected in Section \ref{sec:motivation-works}. 

In this paper, we approach \eqref{eq:H-(Z, pi)} by considering the local monodromy of the two local systems at $t=1$.
The local system $\HH$ has non-trivial local monodromy at $t=1$ (see Section \ref{sec:hgm-ls}). Hence, the isomorphism \eqref{eq:H-(Z, pi)} can hold only if $R^\kappa \pi_{U !}  \CC$ also has non-trivial local monodromy at $t=1$. 
Our first result shows that, if this is the case, then \eqref{eq:H-(Z, pi)} holds: 
\begin{thm}
\label{thm:non-trivial-isomorphism}
Suppose that the local system $R^\kappa \pi_{U !} \QQ$ has non-trivial local monodromy at $t=1$. 
Then the isomorphism \eqref{eq:H-(Z, pi)}
holds.
\end{thm}

Our second result proves that the monodromy assumption of Theorem \ref{thm:non-trivial-isomorphism} holds for families  $(Z, \pi)$ of curves or of even-dimensional varieties: 
 
\begin{thm}
\label{thm:curves-even}
Assume that $(Z, \pi)$ is:
\begin{enumerate}
    \item a family of curves, or
    \item a family of even-dimensional varieties.
    \end{enumerate}
    Then the local monodromy of the local system $R^{\kappa} {\pi_U}_{ !}  \QQ$ at $t=1$ is non-trivial. 
\end{thm}
Therefore, if $(Z, \pi)$ is a family of curves or of even-dimensional varieties, then \eqref{eq:H-(Z, pi)} holds.

Directly verifying the monodromy assumption of Theorem \ref{thm:non-trivial-isomorphism} for families $(Z, \pi)$ of varieties of odd dimension $\kappa>1$ is an interesting problem on its own and requires new ideas.

We sketch the main steps of the proofs of the two theorems in Section \ref{sec:notes-proof}. 

\subsection{Motivation and relation to other works} \label{sec:motivation-works}
Irreducible hypergeometric local systems are motivic by the work of André and Katz \cite{Andre-G-functions, Katz}. 
This means that they arise as subquotients of the variation of cohomology of families of algebraic varieties.
When defined over $\QQ$, they also support a pure rational variation of Hodge structure (VHS). 
This was first conjectured by Corti and Golyshev \cite{AG} and then proven by Fedorov \cite{Fedorov}.

The isomorphism \eqref{eq:H-(Z, pi)} considered in this paper yields a realisation of $\HH$ in terms of the family $(Z, \pi)$. This is particularly appealing as, unlike previous results, \eqref{eq:H-(Z, pi)} identifies $\HH$ with the (pure) variation of the \emph{entire top graded weighted piece} of the middle compactly supported cohomology of the family. 

The family $(Z, \pi)$ and its connection to the parameters $\alpha, \beta$ appear extensively in the literature. 
The family is a specialisation of a family of Gelfand--Kapranov--Zelevinsky (GKZ)'s circuits \cite{GKZ}. 
The work \cite{BCM} relates the finite analogs of the hypergeometric function \mbox{$F(\alpha, \beta\mid t)$} (see \cite[Definition 1.1]{BCM})  with the point count over finite fields of a suitable fibrewise compactification of $(Z, \pi)$. Their result can be regarded as
the arithmetic counterpart of isomorphism \eqref{eq:H-(Z, pi)} at the level of point count. 
Further, in \cite{FRV-duke}, Rodriguez Villegas showed that the Hodge numbers of the stalks of $\gr_\kappa^W PR^\kappa \pi_{U !} \CC$ agree with those of $\HH$ as given by the formula in \cite{Fedorov} (see Section \ref{sec:hodge}), providing further evidence for \eqref{eq:H-(Z, pi)}.  
The paper \cite{RV} develops a systematic framework for ``hypergeometric motives" defined over $\QQ$, which unifies these aspects.
Finally, we conclude this discussion by mentioning the connection to mirror symmetry. We focus on the case of Fano varieties.

To a quasismooth and wellformed Fano weighted complete intersection $X=X_{d_1, \dots, d_c} \subset \PP(a_1, \dots, a_m)$ of Fano index $\lambda=\sum a_i -\sum d_j$ and dimension $n$, one can associate the gamma vector
\[
\gamma_X=(-d_1, \dots, -d_c, -\lambda, a_1, \dots, a_m)
\]
On the one hand, a specialisation of the regularised quantum operator of $X$ is expected to be the pullback along $t=s^{-\lambda}$ of the hypergeometric operator $H$ associated to $\gamma_X$. We refer the reader to \cite{T-stacks, Wang, AlessioG} for known results in this direction. 
On the other hand,  the Givental's  $(n+c)$-dimensional Landau--Ginzburg model $(Y,w)$ associated to $X$ \cite{GIV3} is the pullback along $t=s^{-\lambda}$ of the family $(Z, \pi)$ associated $\gamma_X$. 
Therefore, for gamma vectors arising from Fano weighted complete intersections, the isomorphism \eqref{eq:H-(Z, pi)} can be viewed as a Hodge--theoretic mirror statement relating $X$ and $(Y,w)$. 

The first explicit formulation of the isomorphism \eqref{eq:H-(Z, pi)} is in the aforementioned paper \cite{AG}, which is motivated by mirror symmetry.
In the language of the present paper, the authors of \cite{AG} prove that, for vectors $\gamma=(\gamma_1, \dots, \gamma_l)$ with only one negative entry $\gamma_1$  and such that 
\begin{equation}  \label{eq:wellformed}   \gcd(\gamma_2, \dots ,\widehat{\gamma}_j ,\dots, \gamma_l)=1 \quad \text{ for } j=2, \dots,l \end{equation}
the isomorphism \eqref{eq:H-(Z, pi)} holds \cite[Theorem 1.1]{AG}. 
The restriction \eqref{eq:wellformed} is taken to ensure that $\gamma$ defines a well-formed weighted projective space $\PP=\PP(\gamma_2, \dots, \gamma_l)$.
The proof of \cite[Theorem 1.1]{AG} is based on previous work by Golyshev \cite{Vasily} and a computation of the Hodge numbers of $\gr_\kappa^W PR^\kappa \pi_{U !} \CC$
in terms of the geometry of the affine anticanonical cone of $\PP$. 

In Section \ref{sec:multinomial}, we give a direct proof of  \eqref{eq:H-(Z, pi)} for any gamma vector $\gamma$ such that $\pm \gamma$ has only one negative entry (Theorem \ref{thm:one-negative-gamma})
by exhibiting a non-trivial period of $(Z, \pi)$ 
which is annihilated by the operator $H$. 
Proving the existence of such a period relies on $\gamma$ having only one negative entry. Showing the same for a general gamma vector seems out of reach.  

We note that hypergeometric differential operators associated to general gamma vectors include classical examples, such as the operators annihilating the hypergeometric series 
\[
\sum_{n=0}^\infty \frac{(2n)!^2}{(n!)^4!}\left(\frac{t}{16}\right)^n
\quad\text{or}\quad
\sum_{n=0}^\infty
\frac{(30n)!\,n!}{(6n)!(10n)!(15n)!}
\left(\frac{6^6 10^{10} 15^{15}}{30^{30}}\, t\right)^n 
\]
We also note that, for a general gamma vector, we do not expect the hypergeometric function defined by the corresponding parameters to be a period of the family $(Z,\pi)$. This observation motivates our focus on local systems.  

We conclude by mentioning that, prior to the present paper, the isomorphism \eqref{eq:H-(Z, pi)} was proven \emph{ad hoc} for special cases of gamma vectors with two and three negative entries arising from mirror symmetry \cite{Orbi, AlessioG, ThesisG, ACGGFRV}.

\subsection{Sketch of the proofs}
\label{sec:notes-proof}
As mentioned above, for a general gamma vector, it does not seem possible to construct a period of the family $(Z,\pi)$ which is annihilated by the operator $H$. To prove \eqref{eq:H-(Z, pi)}, we take a different approach. 

To establish Theorem \ref{thm:non-trivial-isomorphism}, we prove that, under the monodromy assumption in the statement, the minimal differential operator of a non-trivial relative $\kappa$-form of weight $\kappa$ is equivalent to $H$. 
This is accomplished in several steps. First, via the theory of GKZ differential systems, we show that distinguished relative forms associated to the family are solutions of \emph{reducible} hypergeometric operators.
This implies that the minimal differential operators of such forms appear as factors of these reducible operators. Next, by a novel analysis of the weight filtration, we show that there exists a relative form of the lowest weight whose minimal differential operator has a genuine singularity at $t=1$. 
Combining these two steps and the equality of ranks of the two sides, we obtain the statement. 

To prove Theorem \ref{thm:curves-even}, we study the singularity of the fibre \(Z_1\). Using the formalism of nearby and vanishing cycles, we obtain the following specialisation sequence in compactly supported cohomology
\begin{equation*}
0 \to H^{d-1}_c(Z_1, \QQ) \to H^{d-1}_c(Z_t, \QQ) \to \QQ \to H^d_c(Z_1, \QQ) \to H^d_c(Z_t, \QQ) \to 0
\end{equation*}
which is compatible with monodromy. Moreover, the action of monodromy on $H^{d-1}_c(Z_t, \QQ)$ is non-trivial if and only if the map $H^d_c(Z_1, \QQ) \to H^d_c(Z_t, \QQ)$ is an isomorphism. 
When $Z_t$ is even-dimensional, we obtain the isomorphism using the fact that the action of monodromy on the Milnor fiber $\QQ$ is non-trivial.
When $Z_t$ is one-dimensional, we obtain the isomorphism from the irreducibility of $Z_1$ and $Z_t$, a non-obvious fact which we prove in Proposition \ref{prop:irreduciblity}.

\subsection*{Structure of the paper}
In Section \ref{sec:families}, we recall general facts about families of affine hypersurfaces in algebraic tori. One of the main goals of the section is to prove that, under a regularity assumption, one can build well-behaved fibrewise compactifications of the family. 
In Section \ref{sec:set-up}, we set up our notation for hypergeometric local systems and the family $(Z, \pi)$, and prove several facts about $(Z, \pi)$ which we use in the proofs of the theorems. 
In Section \ref{sec:multinomial}, we give a direct proof of \eqref{eq:H-(Z, pi)} for gamma vectors with one negative entry. 
In Section \ref{sec:proof-thm-1}, we prove Theorem \ref{thm:non-trivial-isomorphism}. 
In Section \ref{sec:proof-thm-2}, we prove Theorem \ref{thm:curves-even}.

\subsection*{Acknowledgments}
We thank Alessio Corti and Fernando Rodriguez Villegas for attracting our attention to this problem and for many inspiring conversations.
We thank Vadym Kurylenko and Matt Kerr for many helpful discussions. 
We thank Will Sawin for helpful email exchanges regarding nearby and vanishing cycles. 
AA was supported by the Simons Foundation through Award 284558FY19 to the ICTP.
GG acknowledges support from a Simons Investigation Award (n. 929034).

\section{Families of affine hypersurfaces in algebraic tori}
\label{sec:families}
In this section, we study families of affine hypersurfaces in algebraic tori.  

Section \ref{sec:hypersurfaces} collects some standard facts about such hypersurfaces, with \cite{Batyrev-tori} as our main reference. We also introduce a slight generalisation of the classical notion of $\Delta$-regularity \cite[Definition 3.3]{Batyrev-tori}, which we call quasi-$\Delta$-regularity (Definition \ref{def:quasi-Delta}).

In Section \ref{subsec:families}, we turn to families of such hypersurfaces. We first show that families of quasi-$\Delta$-regular hypersurfaces admit well-behaved compactifications (Theorem \ref{thm:compactifying-delta}).
We then focus on the locus of $\Delta$-regularity and collect some key facts about relative cohomology and local systems over the base of the family.

\subsection{Affine hypersurfaces in algebraic tori}
\label{sec:hypersurfaces}
Let $d \geq 1$ be a positive integer. 
Let $L\coloneqq \CC[x_1^{\pm 1},\dots,x_d^{\pm 1}]$ be the ring of Laurent polynomials in $d$ variables over $\CC$. 
We write $x\coloneqq (x_1, \dots, x_d)$ and $x^m \coloneqq x_1^{m_1} \cdots x_d^{m_d}$ for any  $m=(m_1, \dots, m_d) \in \ZZ^d$. Given $f \in L$, we write $Z_f$ for the zero locus of $f$ in the algebraic torus $\TT^d=\Spec L$. 

Let $f = \sum a_m x^{m} \in L$ be a Laurent polynomial, and suppose that $Z_f \subset \TT^{d}$ is a hypersurface.

\subsubsection{Quasi--$\Delta$-regularity}
\label{sec:regularity}
Let $\mathrm{Newt}(f)$ be the Newton polytope of $f$. 
For $F$ a face of $\mathrm{Newt}(f)$, denote by $f_{|F}$ the Laurent polynomial given by the restriction of $f$ to $F$, namely,
\begin{equation}
\label{eq:f-F}
f_{|F} \coloneqq \sum_{m \in F} a_m x^m
\end{equation}

Let $\Delta \subset \RR^d$ be 
a lattice polytope. 
\begin{dfn}{\cite[Definition 3.3]{Batyrev-tori}} 
\label{def:Delta} The polynomial $f$  (and the  
hypersurface $Z_f$) is called \emph{$\Delta$-regular} if $\Delta=\mathrm{Newt}(f)$ and, for each face $F$ of $\Delta$ 
of positive dimension, 
the zero locus $Z_{f_{|F}} \subset \TT^d$ is smooth.  
\end{dfn}

\begin{rem}
Note that, if $\Delta=\mathrm{Newt}(f)$, then, for any face $F$ of $\Delta$ 
of positive dimension, $Z_{f_{|F}}$ is a hypersurface. 
Moreover, if $f$ is $\Delta$-regular, then $Z_f$ is smooth. 
Since, by definition, $f$ can be $\Delta$-regular only with respect to $\Delta=\mathrm{Newt}(f)$, one often drops the reference to $\Delta$ and simply says that $f$ is regular. We use the term $\Delta$-regular to align with \cite{Batyrev-tori}. 
\end{rem}

We introduce a slight generalisation of $\Delta$-regularity: 

\begin{dfn} 
\label{def:quasi-Delta}
We say that  $f$ (and $Z_f$) is \emph{quasi-$\Delta$-regular} if  $\mathrm{Newt}(f)=\Delta$, and, for each \emph{proper} face $F$ of $\Delta$  of positive dimension,
$Z_{f_{|F}} \subset \TT^d$ is smooth. 
\end{dfn}

\begin{rem}
Note that $Z_f$ is $\Delta$-regular if and only if it is quasi-$\Delta$-regular and smooth.
\end{rem}

From now on, we set $\Delta=\mathrm{Newt}(f)$. Moreover, we assume for simplicity that $\Delta$ is full-dimensional: if $\dim(\Delta)=d^\prime < d$, then $Z_f$ is a product  $\TT^{d-d^\prime} \times (Z_{f^\prime} \subset \TT^{d^\prime})$, where $f^\prime$ is $f$ viewed as a polynomial in  $d^\prime$ variables, which brings us back to the full-dimensional case. 
We conclude the section by recalling an alternative characterisation of (quasi-)$\Delta$-regularity. 

It is well known that $\Delta$ determines a proper toric variety $ \mathbb{P}_\Delta$ of dimension $d$ which contains the torus $\TT^d$. We denote by $\bar{Z}_f \subset \mathbb{P}_\Delta$ the closure of $Z_f$ in $\mathbb{P}_\Delta$. The toric variety $\mathbb{P}_\Delta$ is a compactification of  $\TT^d$ by algebraic tori $\TT_F$ which are in one-to-one correspondence with the faces $F$ of $\Delta$. 
The dimension of the torus $\TT_F$  equals the dimension of $F$. The tori $\TT_F$ form a partition of $\mathbb{P}_{\Delta}$:
\begin{equation}
\label{eq:TF}
\Ps_{\Delta}= \bigcup_{F \preceq \Delta } \TT_F
\end{equation}
For each face $F$ of dimension $\ell>0$, the hypersurface $Z_{f_{|F}} \subset \TT^d$ is a product of a torus dimension $d-\ell$ and a hypersurface in an $\ell$-dimensional torus which is isomorphic to  $\bar{Z}_f \cap \ \TT_F \subset \TT_F$. 
Then $f$ is (quasi-)$\Delta$-regular if and only if, for each (proper) face $F$ of $\Delta$ of positive dimension, the hypersurface $\bar{Z}_f \cap \  \TT_F \subset \TT_F$ is smooth.
For further details, see \cite[Sections 2.2 and 3.6] {DK} or \cite[Remark 4.5 and Proposition 4.6]{Batyrev-tori}. 

\begin{rem}
\label{rem:Pdelta-Zf-properties}
The toric variety $\mathbb{P}_\Delta$ is, in general, not smooth. It is if and only if the normal fan $\Sigma_\Delta$ of $\Delta$ is smooth. 
Moreover, the smoothness of $Z_f$ does not, in general, imply that $\bar{Z}_f$ is smooth. However, if $f$ is $\Delta$-regular and $\mathbb{P}_\Delta$ is smooth, then $\bar{Z}_f$ is also smooth. 
More precisely, if $f$ is $\Delta$-regular and $\mathbb{P}_\Delta$ is smooth at a point $p \in \bar{Z}_f$, then $\bar{Z}_f$ is also smooth at $p$. This follows from the description of $\bar{Z}_f$ above and \cite[Lemma 2.3]{DL}. 

More generally, the fan $\Sigma_\Delta$ always admits a smooth refinement $\Sigma$. 
If we denote by $\mathbb{P}$ the smooth toric variety associated to the fan $\Sigma$, then the induced toric morphism $\mathbb{P} \to \mathbb{P}_\Delta$ is a resolution of singularities. 
The closure $V_f\subset \mathbb{P}$ of $Z_f$ has a similar description to the one given above for $\bar{Z}_f \subset \mathbb{P}_\Delta$ (we refer to \cite[Section 3.6]{DK}, or \cite[Section 2.5]{DL}, for more details). 
If $f$ is $\Delta$-regular, then $V_f$ is smooth. Moreover, the boundary divisor $V_f \setminus Z_f$ is a normal crossing divisor. Indeed, since $\Sigma$ is smooth, the boundary divisor $\mathbb{P} \setminus \TT^d$ is a normal crossing divisor and $V_f \setminus Z_f$ is (by $\Delta$-regularity) the transverse intersection of $V_f$ and $\mathbb{P} \setminus \TT^d$. 

Note that if $f$ is quasi-$\Delta$-regular but not $\Delta$-regular, then $V_f$ is no longer smooth, but its singular locus is contained in $Z_f$, and the boundary divisor $V_f \setminus Z_f$ is still a normal crossing divisor.  
\end{rem}

\begin{rem}
The notion of quasi-$\Delta$-regularity is not standard and is introduced for our purposes. Starting from Section \ref{subsec:families}, we will consider families of affine hypersurfaces in algebraic tori that have singular members. By restricting to the locus of quasi-$\Delta$-regularity, the family can be compactified by a relative normal crossing divisor, without introducing singularities in its members (see Theorem \ref{thm:compactifying-delta}). 
\end{rem}

\begin{rem}
\label{rem:ring}
The content of this section remains unchanged if, in the definition of $L$, we replace the field $\mathbb{C}$ with a commutative ring $A$: the torus $\TT^d$ becomes $\Spec A[x_1^{\pm 1},\dots, x_{d}^{\pm 1}]$, and smoothness is interpreted as smoothness over $\Spec A$ \cite{DL}. 
We will require this generality in the proof of Theorem \ref{thm:compactifying-delta}. When $A$ is not specified, we assume $A=\CC$.
\end{rem}

\begin{example}
\label{exa:curve-ft}
Let $d=2$. For $t \in \CC^\times$ fixed, consider the polynomial 
\begin{equation*}
\label{eq:curve}
f_t=-5x_1^2 -2t  x_2^3+ 3tx_1^2x_2^2+4x_1
\end{equation*}
The Newton polytope of $f_t$ is the quadrilateral $\Delta$ in Figure \ref{fig:curve}. One checks that, if $t \neq 1$, the curve $Z_{f_t}$ is $\Delta$-regular. 
The curve $Z_{f_1}$ has an ordinary double point (ODP) at $(1,1)$ and is quasi-$\Delta$-regular.

The toric surface $\mathbb{P}_\Delta$ has three singular points (of type $1/5{(1,3)}, 1/3{(1,1)},$ and $ 1/2{(1,1)}$), corresponding to three of the vertices of $\Delta$. 
For all $t \in \CC^\times$, the boundary $\bar{Z}_{f_t} \setminus Z_{f_t}$ consists of five points.
The curve $\bar{Z}_{f_t}$ does not contain the points of $\mathbb{P}_\Delta$ corresponding to the vertices of $\Delta$ and is smooth at the points of the boundary. 
Thus, if $t \neq 1$, $\bar{Z}_{f_t}$ is smooth, while, if $t=1$, $\bar{Z}_{f_t}$ has a unique ODP at $(1,1)$.
\end{example}

\begin{example}
    \label{exa:nonexa}
    Let $d=3$. Consider the polynomial
    \[
    f=((y-1)^2-(x-1)^3)z-1
    \]
    The Newton polytope of $f$ is the $3$-dimensional simplex $\Delta$ in Figure \ref{fig:nonexa}. 
    Let $F$ be the face of $\Delta$ at height $1$. 
    Then $Z_{f_{|F}} \subset \TT^3$ is 
    the product \[
    \left(((y-1)^2-(x-1)^3=0) \subset \TT^2 \right)\times \TT^1 
    \]
    and the first factor has a cusp at $(1,1)$. Therefore, $f$ is not $\Delta$-regular.
\end{example}

\begin{example} 
\label{exa:chebyshev-ft}
Let $d=3$. For $t \in \CC^\times$ fixed, consider the polynomial 
\begin{equation*}
\label{eq:chebyshev}
f_t=-30 x_1x_2x_3 -\frac{1}{t}+6x_1^5+10 x_2^3 +15 x_3^2
\end{equation*}
The Newton polytope of $f_t$ is the  
$3$-dimensional polytope $\Delta$ in Figure \ref{fig:chebyshev}. One checks that, if $t \neq 1$, the surface $Z_{f_t}$ is $\Delta$-regular. The surface $Z_{f_1}$ has an ODP at $(1,1,1)$ and is quasi-$\Delta$-regular.
\end{example}

\begin{figure}[ht!]
\centering
\begin{tikzpicture}[line cap=round,line join=round,>=triangle 45,x=1cm,y=1cm, scale=0.8]

\clip(-5.755777414580859,-0.9016706575092704) rectangle (7.7258437969507785,3.5492232670187605);
\draw [line width=1pt] (0,3)-- (1,0);
\draw [line width=1pt] (2,0)-- (1,0);
\draw [line width=1pt] (2,2)-- (2,0);
\draw [line width=1pt] (0,3)-- (2,2);
\begin{scriptsize}
\draw [fill=black] (2,0) circle (2.5pt) node[below right] {$(2,0)$}; 
\draw [fill=black] (1,0) circle (2.5pt) node[below left] {$(1,0)$};
\draw [fill=black] (0,3) circle (2.5pt) node[above left] {$(0,3)$};
\draw [fill=black] (2,2) circle (2.5pt) node[below right] {$(2,2)$};
\end{scriptsize}
\end{tikzpicture}
\caption{\label{fig:curve} The Newton polytope of the polynomial $f_t$ in Example \ref{exa:curve-ft}.}
\end{figure}
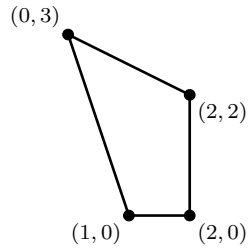

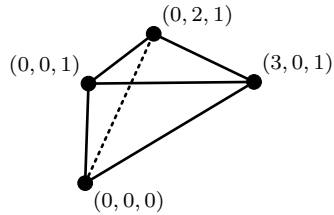
\begin{figure}[ht!]
\centering
\begin{tikzpicture}[line cap=round,line join=round,>=triangle 45,x=1cm,y=1cm, scale=1.1]
\clip(-6.32,-0.25) rectangle (7.32,3.25);
\draw [line width=1pt] (-0.68,1.53)-- (1.3,1.55);
\draw [line width=1pt] (0.1,2.13)-- (1.3,1.55);
\draw [line width=1pt] (-0.68,1.53)-- (0.1,2.13);
\draw [line width=1pt] (-0.68,1.53)-- (-0.72,0.33);
\draw [line width=1pt] (-0.72,0.33)-- (1.3,1.55);
\draw [line width=1pt,dotted] (-0.72,0.33)-- (0.1,2.13);
\begin{scriptsize}
\draw [fill=black] (-0.68,1.53) circle (2.5pt)
node[above left] {$(0,0,1)$};
\draw [fill=black] (1.3,1.55) circle (2.5pt) 
node[above right] {$(3,0,1)$};
\draw [fill=black] (0.1,2.13) circle (2.5pt)
node[above right] {$(0,2,1)$};
\draw [fill=black] (-0.72,0.33) circle (2.5pt) node[below  right] {$(0,0,0)$};
\end{scriptsize}
\end{tikzpicture}
\caption{\label{fig:nonexa} The Newton polytope of the polynomial $f$ in Example \ref{exa:nonexa}.}
\end{figure}

\begin{figure}[ht!]
\centering
\begin{tikzpicture}[line cap=round,line join=round,>=triangle 45,x=1cm,y=1cm, scale=1]
\clip(-3.32,-2.70) rectangle (3.32,1.90);
\draw [line width=1pt] (-0.86,1.27)-- (-0.78,-0.67);
\draw [line width=1pt] (-0.78,-0.67)-- (1.68,-2.23);
\draw [line width=1pt] (1.68,-2.23)-- (1.98,0.13);
\draw [line width=1pt] (-0.86,1.27)-- (1.68,-2.23);
\draw [line width=1pt] (-0.86,1.27)-- (1.98,0.13);
\draw [line width=1pt] (0.48,0.37)-- (1.68,-2.23);
\draw [line width=1pt] (0.48,0.37)-- (1.98,0.13);
\draw [line width=1pt] (-0.86,1.27)-- (0.48,0.37);
\draw [line width=1pt,dotted] (-0.78,-0.67)-- (1.98,0.13);
\begin{scriptsize}
\draw [fill=black] (-0.86,1.27) circle (2.5pt) node[above right] {$(0,0,2)$};
\draw [fill=black] (1.68,-2.23) circle (2.5pt) node[below right] {$(5,0,0)$};
\draw [fill=black] (1.98,0.13) circle (2.5pt) node[below right] {$(0,3,0)$};
\draw [fill=black] (0.48,0.37) circle (2.5pt) node[above right =5pt] {$(1,1,1)$};
\draw [fill=black] (-0.78,-0.67) circle (2.5pt) node[below left] {$(0,0,0)$};
\end{scriptsize}
\end{tikzpicture}
\caption{\label{fig:chebyshev} The Newton polytope of the polynomial $f_t$ in $f_t$ in Example \ref{exa:chebyshev-ft}.}
\end{figure}
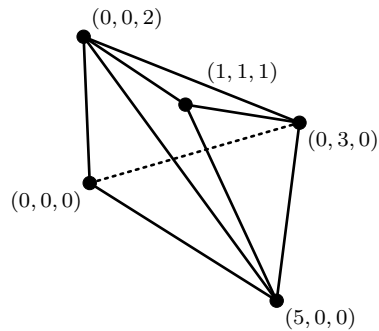

\subsubsection{Cohomology groups}
\label{sec:cohomology}
For a mixed Hodge structure (MHS) $(H, W_{\bullet}, F^{\bullet})$, where $W_{\bullet}$ is the weight filtration and $F^{\bullet}$ is the Hodge filtration, the $i$th graded piece $\operatorname{gr}_i^W H$ is the quotient $\operatorname{gr}_i^W H \coloneqq W_i/W_{i-1}$. It is a pure Hodge structure of weight $i$. 

We recall a few results on the cohomology groups $H^i(Z_f, \QQ)$. 
We write $\kappa \coloneqq \dim Z_f=d-1$. 

\subsubsection*{General facts} 
Assume that $Z_f$ is $\Delta$-regular. Then, 
one has that (\cite[Section 3]{Batyrev-tori}):
\begin{enumerate}
    \item for all $i > \kappa$, $H^i(Z_f, \QQ)=0$; 
    \item for each $i \leq \kappa$, $H^i(Z_f, \QQ)$ carries a MHS
    with weight filtration $W_{j}=W_jH^i(Z_f,\QQ)$:
    \[
     0={W}_{i-1} \subset{W}_{i} \subset \dots \subset W_{2i}= H^i(Z_f,\QQ)
    \] 
    and Hodge filtration $F^{j}=F^jH^i(Z_f,\CC)$:
    \[
    H^i(Z_f,\CC)=F^0 \supset F^1 \supset \dots \supset F^{i+1}=0
    \] 

    \item the pullback map $H^i(\TT^d, \QQ) \to H^i(Z_f, \QQ)$
    is a morphism of MHS. It is bijective for $i < \kappa$ and injective for $ i = \kappa$. 
\end{enumerate}
\smallskip

\noindent It follows from (3) that, for $i<\kappa$,  $H^i(Z_f, \QQ)$ is a  pure Hodge structure of rank $\binom{d}{i}$ and weight $2i$, and $H^i(Z_f, \CC)=F^i H^i(Z_f,\CC)$. The image of $H^{\kappa}(\TT^d,\QQ)$ in $H^\kappa(Z_f, \QQ)$ has rank $d$ and weight $2\kappa$, and the image of $H^{\kappa}(\TT^d,\CC)$ in $H^\kappa(Z_f, \CC)$ is contained in $F^{\kappa}H^\kappa(Z_f, \CC)$.

The dependency of $H^{\kappa}(Z_f, \QQ)$ on $f$ is captured by 
the primitive cohomology $PH^{\kappa}(Z_f, \QQ)$, defined as the cokernel of the pullback map 
$H^{\kappa}(\TT^d, \QQ) \to H^{\kappa}(Z_f, \QQ)$.
By (2) and (3),  $PH^{\kappa}(Z_f, \QQ)$ carries a MHS of weights $\in \{k,\dots,  2k\}$.

In this paper, we will be interested in the graded piece 
\begin{equation*}
\gr_\kappa^W PH^{\kappa}(Z_f, \QQ)=W_\kappa PH^{\kappa}(Z_f, \QQ)
\end{equation*}
which carries a pure Hodge structure of weight $\kappa$.
Note that, if $\kappa > 0$, $W_\kappa H^{\kappa}(Z_f, \QQ)$ and $W_\kappa PH^{\kappa}(Z_f, \QQ)$ coincide while, if $\kappa=0$, the latter is the quotient of the first by a copy of $\QQ$. 

\begin{rem}
\label{rem:compact-support}
Above, we have worked with cohomology, but one may as well work with compactly supported cohomology. 

For $i <\kappa$, the cohomology group $H^i_c(Z_f, \QQ)$ is trivial. For $i \geq \kappa$, $H^i_c(Z_f, \QQ)$ carries a MHS of weights $\leq i$. 
The primitive cohomology with compact support $PH_c^\kappa(Z_f, \QQ)$ is the kernel of the Gysin morphism $H_c^\kappa(Z_f, \QQ) \to H_c^{\kappa+2}(\TT^d, \QQ)$. 

Let $\QQ(-\kappa)$ be the pure Hodge structure of rank $1$ and weight $2\kappa$.
The two cohomology theories are related by Poncair\'{e} duality
\[ H^i(Z_f , \QQ) \times H^{2\kappa -i}_c(Z_f, \QQ) \to 
\QQ(-\kappa)
\]
which is compatible with the MHS.  
\end{rem}

\begin{example}[Example \ref{exa:curve-ft} continued]
\label{exa:curve-numbers}

Fix $t\in \CC^\times \setminus \{ 1\}$ and let $Z_{f_t}, \bar{Z}_{f_t}$ be as in Example \ref{exa:curve-ft}. By Baker's Theorem, since $\Delta$ has two interior points, $Z_{f_t}$ has genus $2$, thus $\gr_1^W H^1(Z_{f_t}, \QQ)$ has rank $4$ and Hodge numbers $h^{1,0}=h^{0,1}=2$.
Moreover, since $\bar{Z}_{f_t}\setminus Z_{f_t}$ consists of five points,  $W_2 H^1(Z_{f_t}, \QQ) = H^1(Z_{f_t}, \QQ) $ has rank $8$ and $\gr_{2}^W H^1(Z_{f_t}, \QQ)$ 
has Hodge numbers $h^{2,0}=h^{0,2}=0$, $h^{1,1} =4$.
\end{example}

\begin{example}[Examples \ref{exa:chebyshev-ft} continued]
\label{exa:chebyshev-hodge}
Fix $t\in \CC^\times \setminus \{ 1\}$, and let $Z_{f_t}$ be as in Example \ref{exa:chebyshev-ft}. The cohomology group $H^{2}(Z_{f_t},\QQ)$ carries a MHS of weights
$2,3$, and $4$.
The graded piece $\gr_2^W H^{2}(Z_{f_t},\QQ)$ 
has rank $8$ and Hodge numbers $h^{2,0}=h^{0,2}=0$,  $h^{1,1}=8$.
For a full description of the MHS on $H^{2}(Z_{f_t},\QQ)$, we refer the reader to  \cite[Section 4]{FRV-duke}.
\end{example}
\smallskip

The middle cohomology $H^\kappa(Z_f, \QQ)$ is related to the cohomology group $H^d(\TT^d \setminus Z_f, \QQ)$ of the complement $\TT^d \setminus Z_f$ 
via the exact sequence of MHS:
\begin{equation}
\label{eq:residue-map}
0 \to H^d(\TT^d,\QQ) \to H^{d}(\TT^d \setminus Z_f,\QQ) \xrightarrow{res} H^{\kappa}(Z_f,\QQ)(-1) \to  0
\end{equation}
The map $res$ is called the Poincar\'e residue mapping \cite[Section 5]{Batyrev-tori}.

The complement $\TT^d \setminus Z_f$ is isomorphic to the hypersurface $Z_{\widetilde{f}} \subset \TT^{d+1}$ defined by the Laurent polynomial \[\widetilde{f}=x_0f-1 \in L[x_0^{\pm 1}]=\CC[x_0^{\pm 1}, \dots, x_d^{\pm 1}]\]
The Laurent polynomial $f$ is $\Delta$-regular if and only if $\widetilde{f}$ is $\widetilde{\Delta}$-regular, where $\widetilde{\Delta} \subset \RR^{d+1}$ is  the Newton polytope of $\widetilde{f}$ (\cite[Proposition 4.2]{Batyrev-tori}). 
The image of the pullback map $H^{d}(\TT^{d+1}, \QQ) \to H^d(Z_{\widetilde{f}},\QQ)$ is spanned by the $d+1$ differential forms:
\[
\frac{dx_1}{x_1} \wedge \cdots \wedge \frac{dx_d}{x_d}  \quad \text{and} \quad  \frac{df}{f} \frac{dx_1}{x_1} \wedge \cdots\wedge\widehat{\frac{dx_i}{x_i}}\wedge \dots  \wedge \frac{dx_d}{x_d}  \quad \text{for} \ i=1, \dots, d
\]
Using \eqref{eq:residue-map} one finds that  \mbox{$F^0 H^{d}(\TT^d \setminus Z_f,\CC)=F^1 H^{d}(\TT^d \setminus Z_f,\CC)$} and $W_d H^{d}(\TT^d \setminus Z_f,\QQ)=0$. 
\smallskip

The maps described above fit into the commutative diagram of MHSs: 
\begin{equation}
    \label{eq:Z-Ztilde-diagram}
\begin{tikzcd}
 &  &  0 \arrow[d] & 0 \arrow[d]   & \\
 0 \arrow[r] & H^{d}(\TT^d, \QQ) \arrow[r] \arrow[d, "\simeq"] &  H^{d}(\TT^{d+1}, \QQ)  \arrow[r, "res"] \arrow[d] & H^\kappa(\TT^d, \QQ)(-1) \arrow[d]  \arrow{r} & 0 \\
0 \arrow[r] & H^{d}(\TT^d, \QQ) \arrow[r]  &  H^{d}(\TT^d \setminus Z_{f}, \QQ) \arrow[r, "res" ] \arrow[d] & H^{\kappa}( Z_{f}, \QQ)(-1) \arrow[d] \arrow{r} & 0 \\
&       &  PH^{d}(\TT^d \setminus Z_f, \QQ) \arrow[r, " \simeq "]      \arrow[d]      &  PH^{\kappa}(Z_f, \QQ)(-1) \arrow[d]    & \\
 &      & 0           & 0 & 
\end{tikzcd}
\end{equation}
Note that, if $\kappa>0$, 
$W_{d+1}H^d (\TT^d \setminus Z_f,\QQ)$ is isomorphic to $W_{\kappa}PH^\kappa (Z_f,\QQ)(-1)$,  while, if $\kappa=0$, the latter 
is the quotient of the first one by the rank-$2$ image of 
 $H^{1}(\TT^{2}, \QQ) \to H^1(\TT^1\setminus Z_f,\QQ)$.

\subsubsection*{Dwork--Katz method}
Assume that $f$ is $\Delta$-regular. 
Recall that $H^d(\TT^d \setminus Z_f, \QQ)\otimes \CC$ is isomorphic to the de Rham cohomology group $H^d_{dR}(\TT^d \setminus Z_f)$. 
The latter admits an explicit combinatorial description, given in \cite[Sections 7-8]{Batyrev-tori} and known as the Dwork--Katz method. 
We briefly review it. 

Let $S_\Delta^+$ be the $\CC$-subalgebra of $L[x_0]$ generated as a $\CC$-vector space by all monomials $x_0^kx^m$ such that  $k \ge 1$  and $m/k \in \Delta.$ 
For each integer $\ell$, let $\mathcal{E}^{-\ell}$ be the $\CC$-vector subspace of $S_\Delta^+$  generated by monomials $x_0^k x^m$ with $k\le \ell$, and $\mathcal{I}^{\ell}$ be the ideal of $S_\Delta^+$ generated by monomials $x_0^kx^m$ such that $m/k$ does not lie on any face of $\Delta$ of codimension $\ell$. 

The monomials in $S_{\Delta}^{+}$ correspond to the nonzero integral points of the cone
\[ C(\Delta) \coloneqq \{\lambda \cdot (1, u) \mid \lambda \geq 0, u \in \Delta\ \} \subset \RR^{d+1}
\] 
The monomials in $\mathcal{E}^{-\ell}$ correspond to the integral points $(k,m)$  of $C(\Delta)$ with $k \le \ell$, while the monomials in $\mathcal{I}^{\ell}$ correspond to the integral points of $C(\Delta)$ that do not lie on any face of $C(\Delta)$ of codimension $\ell$. 

On $S_\Delta^+,$ define the operators 
\begin{equation*}
\label{eq:op-Di}
\mathcal{D}_i := x_i\frac{\partial}{\partial x_i} + \widetilde{f}_i \qquad i=0,\dots,d
\end{equation*}
where  $\widetilde{f}_i \coloneqq x_i \frac{\partial}{\partial x_i} \widetilde{f}.$ 

Let $\Omega^d(\TT^d \setminus Z_f)$ be the space of meromorphic differential forms on $\TT^d$ with  poles  of arbitrary order along $Z_f$, and define the map 
\begin{align*}
    \mathcal{R}: S_\Delta^+ &\to  \Omega^d(\TT^d \setminus Z_f)\\
    x_0^k x^m &\mapsto (-1)^{k-1} (k-1)!  \; \frac{x^m }{f^k } \frac{dx}{x}
\end{align*}
where $\frac{dx}{x}\coloneqq \frac{dx_1}{x_1}\wedge\dots \wedge \frac{dx_d}{x_d}$.

\begin{thm}{\cite[Theorem 7.13, Theorem 8.1, Theorem 8.2]{Batyrev-tori}}
\label{thm:Baty-summary} \hfill
\begin{enumerate}
    \item The composition 
    $S_\Delta^+ \xrightarrow{\mathcal{R}} \Omega^d(\TT^d\setminus Z_f) \to H^d(\TT^d\setminus Z_f, \CC)$
    factors through the quotient $S_\Delta^+/ \bigoplus_{i=0}^d \mathcal{D}_i S_\Delta^+$, and the map
    \[
    \rho: S_\Delta^+/ \bigoplus_{i=0}^d \mathcal{D}_i S_\Delta^+ \to H^d(\TT^d\setminus Z_f, \CC)
    \] is an isomorphism. Below, we denote the image of a subset of $S_{\Delta}^+$ in $S_\Delta^+/  \bigoplus_{i=0}^d \mathcal{D}_i S_\Delta^+$ by the same symbol as the subset itself.
    \item The Hodge filtration $F^j$ on $H^d(\TT^d \setminus Z_f,\CC)$ fits into the commutative diagram:
    \begin{equation*}
    \begin{tikzcd}
    0=F^{d+1}\ar[r, hook] \ar[d,equal] & F^d \ar[r, hook]\ar[d,equal] & \dots \ar[r,hook]  &F^{1} = F^0 
    = H^d(\TT^d \setminus Z_f, \CC)
    \ar[d,equal] 
    \\
    \rho(\mathcal{E}^0) \ar[r, hook] & \rho(\mathcal{E}^{-1}) \ar[r,hook] & \dots  \ar[r,hook] & \rho(\mathcal{E}^{-d})
    \end{tikzcd} 
    \end{equation*}
    \item The weight filtration $W_j$ on $H^d(\TT^d \setminus Z_f,\CC)$ fits into the following commutative diagram: 
    \begin{equation*} 
    \begin{tikzcd}
    0=W_d\ar[r, hook]  & W_{d+1} \ar[r, hook] \ar[d,equal]  & \dots \ar[r, hook] &W_{2d-1} \ar[d,equal]\ar[r,hook]  & W_{2d} = H^d(\TT^d \setminus Z_f,\CC) \\
    & \rho(\mathcal{I}^{1}) \ar[r,hook] & \dots  \ar[r,hook] & \rho(\mathcal{I}^{d-1})   & 
    \end{tikzcd} 
    \end{equation*}
    \end{enumerate}
\end{thm}

\smallskip

\begin{example}[$d=2$]
Let $d=2$. Then $f$ is a two-variable polynomial and $\Delta$ is a polygon.  We have
\[  H^2(\TT^2\setminus Z_{f}, \CC)=F^0=F^1 \supset F^2 \supset F^3=0
\]
and 
\[
0=W_2 \subset W_3 \subset W_4=H^2(\TT^2\setminus Z_{f_t}, \CC)\]
The nonzero integral points $(m,k)$ of $C(\Delta)$ such that $k \leq 2$ generate $H^2(\TT^2\setminus Z_{f_t}, \CC)=F^0=F^1$, with those such that $k=1$ generating $F^2$. 
The integral interior points of $C(\Delta)$ generate $W_3$. 
\end{example}

\begin{example}[$d=3$]
   Let $d=3$. Then $f$ is a three-variable polynomial and $\Delta$ is a $3$-dimensional.  We have
\[  H^3(\TT^3\setminus Z_{f}, \CC)=F^0=F^1 \supset F^2 \supset F^3 \supset F^4=0
\]
and 
\[
0=W_3 \subset W_4 \subset W_5 \subset W_6=H^3(\TT^3\setminus Z_{f_t}, \CC)\]
The nonzero integral points $(m,k)$ of $C(\Delta)$ such that $k \leq 3$ generate $H^3(\TT^3\setminus Z_{f_t}, \CC)=F^0=F^1$, with those such that $k\leq2$ generating $F^2$ and those such that $k=1$ generating $F^3$. 
The integral interior points of $C(\Delta)$ generate $W_4$. The integral points in the interior of $C(\Delta)$ or in the relative interior of its $3$-dimensional faces generate $W_5$. 
\end{example}

\begin{example}[Examples \ref{exa:curve-ft} and \ref{exa:curve-numbers} continued]
\label{exa:curve-cone} Fix $t\in \CC^\times \setminus \{ 1\}$, let $f_t$ be the polynomial of Example \ref{exa:curve-ft}, and consider $H^2(\TT^2 \setminus Z_{f_t}, \CC)$.

The cone $C(\Delta)$ has two interior integral points $(m,k)$ such that $k=1$, namely $(1,1,1)$ and $(1,2,1)$. The intersection $W_3 \cap F^2$ is generated by the corresponding two elements. Since $h^{1,0}( \gr_1^WH^1(Z_{f_t},\CC))=2$, the two elements must be independent. 

A computer calculation shows that the classes of the forms 
\begin{align*}
    \omega_1  = \frac{x_1x_2}{f} \frac{dx}{x} \qquad \omega_2  = \frac{x_1x_2^2}{f} \frac{dx}{x}\\
    \omega_3  = \frac{x_1^2x_2^2}{f^2} \frac{dx}{x} \qquad \omega_4  = \frac{x_1^2x_2^4}{f^2} \frac{dx}{x}
\end{align*}
form a basis of $W_3$.
Moreover, $W_4=H^2 (\TT^2\setminus Z_{f_t}, \CC)$ is generated by the classes of $\omega_1, \dots, \omega_4$ and the five classes of
\[
\frac{x_1}{f}\frac{dx}{x} \quad   \frac{x_1^2x_2^2}{f}\frac{dx}{x} \quad  \frac{x_1^2 x_2}{f}\frac{dx}{x} \quad \frac{x_1^2}{f}\frac{dx}{x} \quad\frac{x_2^3}{f}\frac{dx}{x}
\]
The one-dimensional kernel of the residue map $H^2(\TT^2 \setminus Z_{f_t}, \CC)\to H^1(Z_{f_t}, \CC)(-1)$ is generated by the class of
\[
\frac{dx}{x } = -5 \frac{x_1^2}{f}\frac{dx}{x} -2t \frac{x_2^3}{f}\frac{dx}{x} +3t \frac{x_1^2x_2^2}{f}\frac{dx}{x}  +4 \frac{x_1}{f}\frac{dx}{x}  
\]

\end{example}

\begin{example}[Examples \ref{exa:chebyshev-ft} and \ref{exa:chebyshev-hodge} continued]
\label{exa:chebyshev-cone}
Fix $t\in \CC^\times \setminus \{ 1\}$, let $f_t$ be the polynomial of Example \ref{exa:chebyshev-ft} and   consider $H^2(\TT^2 \setminus Z_{f_t}, \CC)$.

The cone $C(\Delta)$ has no interior points $(m,k)$ with $k=1$. Hence $W_4 \cap F^3=0$ (and $W_4 \cap F^2=W_4 \cap F^1$). This 
agrees with the fact that  $h^{2,0}(\gr_2^W H^{2}(Z_{f_t},\CC))=0$. However, $C(\Delta)$ has $15$  integral interior points $(m,k)$ with $k=2$. The piece $W_4$ is spanned by the corresponding elements. 

We already know that $W_4$ has rank $8$ as $\gr_2^WH^2(Z_{f_t}, \CC)$ has rank $8$. A computer calculation shows that 
the classes of the forms
\begin{align*}
   \omega_1= \frac{x_1^2x_2x_3^2}{f^2}\frac{dx}{x}  \quad  \omega_2= \frac{x_1^2x_2x_3}{f^2}\frac{dx}{x} \quad \omega_3= \frac{x_1^4x_2x_3}{f^2}\frac{dx}{x}  \quad  \omega_4 = \frac{x_1x_2x_3^2}{f^2}\frac{dx}{x} \\
   \omega_5= \frac{x_1x_2x_3}{f^2}\frac{dx}{x}  \quad  \omega_6= \frac{x_1^3x_2x_3^2}{f^2}\frac{dx}{x} \quad \omega_7= \frac{x_1^3x_2x_3}{f^2}\frac{dx}{x}  \quad  \omega_8 = \frac{x_1x_2^2x_3}{f^2}\frac{dx}{x} 
\end{align*} form a basis of $W_4$. 
\end{example}

\subsection{Families of hypersurfaces} \label{subsec:families}
Let $d, l \geq 1$ be integers. Let $\bs  f  \in \CC[x_1^{\pm},\dots,x_d^{\pm1}, u_1^{\pm 1},\dots, u_l^{\pm 1}]$ be a Laurent polynomial. 
Let $\bs U$ be an affine open subset of  $\TT^l=\Spec[u_1^{\pm 1}, \dots, u_1^{\pm l}]$.  Let $Z_U$ be the zero locus of ${\bs f}$ in $\TT^d \times \bs U$, that is, $Z_U \coloneqq Z_f \cap \; \TT^d \times U$, and let $\pi_U$ 
be the projection map
\begin{align}
\label{eq:projection}
     \bs{\pi_U}: \bs{Z_U} &\to \bs{U}\\
        (x_1, \dots, x_d,u_1, \dots, u_l) &\mapsto  (u_1, \dots, u_l)
\end{align}
For a point $u=(u_1, \dots, u_l) \in \bs{U}$  denote by $\bs{Z}_u$ the fiber of $\pi_{U}$ over $u$. Then $Z_u$ is the zero locus $Z_{\bs{f}_u} \subset \TT^d$ of the Laurent polynomial $\bs{f}_u\coloneqq \bs{f}(\cdot,u) \in \CC[x_1^{\pm 1},\dots,x_d^{\pm 1}]$. 
We assume that $Z_u \subset \TT^d$ is a hypersurface for all $u \in U$.
We view the map ${\pi_U}$ as a family of affine hypersurfaces $(Z_u)_{u \in U}$ in the algebraic torus $\mathbb{T}^{d}$.  

For the rest of this section, we drop the subscript $U$ and simply write  $Z=Z_U$ and $\pi=\pi_U$.

\subsubsection{Fibrewise compactifications} 
\label{sec:fibre-cmpt} 
Studying the map $\pi$ and the variation of cohomology of $Z_u$ is challenging for two main reasons. 
First, $\pi$ is not necessarily smooth, although this can be fixed by restricting to the regular locus. 
Second, it is not proper, which is undesirable (for instance, when $\pi$ is smooth, properness ensures that the rank of $H^i(Z_u, \QQ)$ is locally constant).
We discuss how restricting to the locus of quasi-$\Delta$ regularity allows us to obtain well-behaved fibrewise compactifications of $\pi$.

We denote by $\Delta$ the Newton polytope of $\bs f$ over $\CC[u_1^{\pm1},\dots, u_l^{\pm1}]$.
\begin{thm}\label{thm:compactifying-delta}
    Assume that 
    \begin{enumerate}
        \item $\bs{Z}$ is smooth;
        \item for each $u \in \bs{U}$ the fibre $\bs{Z}_u$ is quasi-$\Delta$-regular.
    \end{enumerate}
    Then there exists a fiberwise compactification $\bar{\bs\pi}: \bs{V} \to \bs{U}$ of $\bs \pi$, that is, a partial compactification $ \bs V$ of $\bs Z$ and  a commutative diagram  
    \[
\begin{tikzcd}
\bs Z  \ar[r,hook, "\bs j" ] \ar[rd, "\bs\pi" ] & \bs V \ar[d, "\bar{\bs \pi}"] \\
 & \bs U 
\end{tikzcd}
\] where $\bar{\bs \pi}$ is proper and $\bs j : \bs Z \to \bs V $ denotes the open inclusion such that
    \begin{enumerate}
    	\item[(i)] $\bs{V}$ is smooth; 
        \item[(ii)] $\bs V \setminus  \bs j(\bs Z)$ is a normal crossing divisor over $\bs{U}$;
        \item[(iii)] if $x\in {\bs V}_u:={\bar{\bs \pi}}^{-1}(u)$, for some $u$, is a critical point of $\bar{\bs \pi}$, then $x \in \bs j(\bs Z_u)$.
    \end{enumerate} 
\end{thm}

\begin{proof}
The conditions of the theorem imply that  $\bs f$ is quasi-$\Delta$-regular over $\OO_{\bs{U}}(\bs{U})$. Indeed, let $F$ be a 
proper 
face of $\Delta$. Consider the zero locus $(\bs f_{|F}=0) \subset \TT^d \times \bs U$, and the projection map 
\[{\bs \pi}_F \colon (\bs f_{|F}=0) \to \bs{U}\] 

First, we show that the map ${\bs \pi}_F$ is flat. 
By miracle flatness, it suffices to prove that the domain is smooth, the target is smooth, and the fibres have the same dimension. 
Indeed, $U$ is smooth. 
Moreover, (2) implies that all the fibres have the same dimension and are smooth, thus the domain is smooth. 
Now, since ${\bs \pi}_F$ is flat, and the domain and fibres are smooth, we have by \cite[Theorem 10.2]{Hart} that ${\bs \pi}_F$ is smooth. 
This proves quasi-$\Delta$-regularity of $\bs f$ over $\OO_{\bs{U}}(\bs{U})$.

Now, let $\mathbb{P}$ be the toric variety over $\bs{U}$ associated to a smooth refinement of the normal fan of $\Delta$. Let $\bs V$ be the closure of $\bs Z$ in $\mathbb{P}$, and let $\bar{\bs \pi} \colon \bs V \to \bs U$ be the projection map. Then, $\bar{\bs \pi}$  is a fibrewise compactification of $\bs \pi$. 
Quasi-$\Delta$-regularity of $\bs Z$ over $\bs U$ implies that for each $z \in \bs V \setminus j(Z)$, $V$ is smooth over $\bs U$ at $z$ (hence smooth over $\CC$ at $z$) and $V\setminus Z$ is a normal crossing divisor over $\bs U$. Moreover, condition (2) implies (iii), see Remark \ref{rem:Pdelta-Zf-properties}. 
Finally, since $Z$ is smooth, we have that $V= \left( V\setminus Z\right) \cup Z$ is smooth.
\end{proof}

\subsubsection{Relative cohomology}\label{subsec:relative-coho}  
Assume that $Z$ and $U$ satisfy the assumptions of Theorem  \ref{thm:compactifying-delta} and that the projection map $\pi \colon Z \to U$ in \eqref{eq:projection} is smooth. Equivalently, assume that, for all $u \in U$, $Z_u$ is $\Delta$-regular. 

Let $\widetilde{f}=x_0f-1 \in \CC[x_0^{\pm 1}, x_1^{\pm 1},\dots,x_d^{\pm 1}, u_1^{\pm 1},\dots, u_l^{\pm 1}]$, let $\widetilde{Z}$ be the zero locus of $\widetilde{f}$ in $\TT^{d+1} \times U$, and let $\widetilde{\Delta}$ be the Newton polytope of $\widetilde{f}$ over $\CC[u_1^{\pm1},\dots, u_l^{\pm1}].$ 
By the fact that $\widetilde{f}$ is $\tilde{\Delta}$-regular and an application of the Jacobian criterion to $\widetilde{f}$, one checks that $\widetilde{f}$ and $U$ also satisfy the assumptions of Theorem \ref{thm:compactifying-delta} and that the projection map $\widetilde{\pi} \colon \widetilde{Z} \to U$ is smooth. 

For a morphism $X \to U$, let \( \mathcal{H}^n(X/U) \) be the relative algebraic de Rham cohomology sheaf of $X$ over $U$ of degree $n$. 
The sheaf $\mathcal{H}^{n}({Z}/U)$ has stalk at $u \in U$ the $n$th algebraic de Rham cohomology of $Z_u$. 
It is equipped with a connection $\nabla: \mathcal{H}^n({Z}/U)\to  \mathcal{H}^n({Z}/U) \otimes \Omega^1_U$ called the algebraic Gauss-Manin  connection \cite{Katz-Oda}. This connection makes $\mathcal{H}^n({Z}/U)$ into a $\mathcal{D}_U$-module, where $\mathcal{D}_U$ 
is the sheaf of differentials over $U$,  by defining 
\[ \frac{\partial}{\partial u_j} \cdot \omega\coloneqq \nabla_{\frac{\partial}{\partial u_j}} \omega
\]
where $j=1, \dots, l$ and $\omega$ is a section of $\mathcal{H}^n({Z}/U)$.
Moreover, $\mathcal{H}^n({Z}/U)$ carries a variation of mixed Hodge structure (VMHS) which, on stalks, yields the MHSs of the fibres. 
The analog statements hold for $\mathcal{H}^{n}(\widetilde{Z}/U)$. 

We will only consider the non-trivial cases $\mathcal{H}^{\kappa}({Z}/U)$ and $ \mathcal{H}^{d}(\widetilde{Z}/U)$.
For a torus $\TT^n$, we will write $\mathcal{H}^i(\TT^n \times U/U)$ for the $i$th relative algebraic de Rham cohomology sheaf of $\TT^n \times U$ over $U$ (with trivial connection). We will denote by $P\mathcal{H}^n({Z}/U)$ and $P\mathcal{H}^{d}(\widetilde{Z}/U)$ the relative primitive algebraic de Rham cohomology
sheaves. These sheaves fit into a commutative diagram of VMHSs, compatible with the connections, which, on stalks, restricts to the diagram \eqref{eq:Z-Ztilde-diagram}.

The sections of $\mathcal{H}^d(\widetilde{Z}/U)$ are generated by the forms 
\begin{equation}
\label{eq:form-omegab}
\omega_{(\beta_0,\beta)}= \frac{x^{\beta}}{f^{\beta_0}} \frac{dx}{x}
\end{equation}
where $\beta_0 \in \Z_{\ge 1}$, $\beta=(\beta_1,\dots,\beta_d) \in \Z^d$,  and $\beta/\beta_0 \in \Delta$. We call such forms \emph{monomial forms}.
The action of the algebraic Gauss-Manin connection $\nabla$ 
is given by

\begin{align}
\label{eq:construction-of-gauss-manin}
\nabla_{\frac{\partial}{\partial u_j}}\omega_{(\beta_0,\beta)} 
   & =  \frac{\partial}{\partial u_j} \frac{x^{\beta}}{f^{\beta_0}} \frac{dx}{x}  
\end{align} 
We refer to \cite[Section 11]{Batyrev-tori} for more details. 
The weight filtration $W_j\mathcal{H}^d(\widetilde{Z}/U)$ and the Hodge filtration $F^j\mathcal{H}^d(\widetilde{Z}/U)$ admit the description of Theorem 
\ref{thm:Baty-summary}.

The following theorem relates the forms in Equation \eqref{eq:form-omegab} with Gelfand-Kapranov-Zelevinsky (GKZ) systems of differential equations \cite{GKZ}.  

\begin{thm}{\cite[Theorem 14.2]{Batyrev-tori}}\label{theo:form-GKZ}
Let \[f= \sum_{j=1}^l u_j x^{m_j}\] for $m_j =(m_{1j},\dots,m_{dj}) \in \Z^d$. Let $\omega_{(\beta_0, \beta)}$ be as in Equation \eqref{eq:form-omegab}.
Then the \textit{periods} of $\omega_{(\beta_0,\beta)}$ satisfy the GKZ system of differential equations:
\begin{equation}\label{eq:form-GKZ}
\left\{
\begin{aligned}
& \sum_{j=1}^l \theta_j \; \Phi(u) = -\beta_0 \; \Phi(u),\\
& \sum_{j=1}^l m_{ij} \theta_j \; \Phi(u) = -\beta_i \; \Phi(u)
\quad (i=1,\dots,d),\\
& \prod_{r_j>0} {\partial_j}^{r_j} \; \Phi(u)
= \prod_{r_j<0}{\partial_j}^{-r_j} \;  \Phi(u)
\quad \text{for } r \in R.
\end{aligned}
\right.
\end{equation}
where $\theta_j \coloneqq u_j \frac{\partial }{\partial u_j}$  and 
$\partial_j \coloneqq \frac{\partial }{\partial u_j}$, 
and $R$ is the lattice of integral relations
\begin{equation}
R \coloneqq \left \{ r=(r_1,\dots,r_l) \in \Z^l \ \middle| \  \sum_{j=1}^l r_j(1,m_{1j} \dots, m_{dj})=0 \right\}
\end{equation}
\end{thm}

\smallskip

\begin{example}[Examples \ref{exa:curve-ft}, \ref{exa:curve-numbers}, \ref{exa:curve-cone} continued]
\label{exa:curve-u}
Let $d=2$ and $l=4$. Consider the polynomial \begin{equation}
\label{eq:curve-bigF}
\end{equation}
 Its zero locus in $\TT^2 \times (\CC^\times)^4$ is smooth. Its Newton polytope over $\CC[u_1^{\pm1}, u_2^{\pm1}, u_3^{\pm1}, u_4^{\pm1}]$ is the quadrilateral $\Delta$ in Figure \ref{fig:curve} (the polynomial $f_t$ of Example \ref{exa:curve-ft} is obtained by setting $u_1=-5, u_2=-2t, u_3=3t, u_4=4$).
The $\Delta$-regularity locus of the projection over $u$ is given by 
\begin{equation}
\label{eq:curve-value}
\frac{u_3^{3}u_4^{4}}{u_1^5u_2^2} \neq -\frac{3^32^5 }{5^5}
\end{equation} 
The lattice $R$ of integral relations of Theorem \ref{theo:form-GKZ} is spanned by $\gamma=(-5,-2,3,4) \in \ZZ^4$.

\end{example}

\begin{example}
[Examples \ref{exa:chebyshev-ft}, \ref{exa:chebyshev-hodge}, \ref{exa:chebyshev-cone} continued]
    \label{exa:chebyshev-u} The lattice $R$ of integral relations of the five vectors $(1, m_j) \in \ZZ^4$, with $m_j\in \ZZ^3$ the vertices of the polytope $\Delta$ of Example \ref{exa:chebyshev-ft},  is generated by $\gamma=(-30,-1,6,10,15)\in \ZZ^5$.
\end{example}

\subsubsection{Local systems}  \label{sec:local-systems}
We work under the assumptions of Section \ref{subsec:relative-coho}. We consider the local systems $R^n {\pi}_{*}\QQ$ and $R^n{\pi}_{!}\QQ$, or, equivalently, their monodromy representations.  
These sheaves are indeed local systems under our assumptions, see \cite[Proposition 6.14]{Deligne-equations}. 

The stalk of $R^n {\pi}_{*}\QQ$, respectively $R^n {\pi}_{!}\QQ$,  at a point $u$, is isomorphic to the $n$th cohomology, respectively compactly supported $n$th cohomology, with rational coefficients of the fibre of $\pi$ at $u$. 
The local systems  $R^n {\pi}_* \QQ$ and $R^{2\kappa-n} {\pi}_!\QQ$ are dual to each other, with the duality given on stalks by Poincaré duality for the cohomology of fibres.  
Both $R^n {\pi}_* \QQ$ and $R^n {\pi}_! \QQ$ carry VMHSs which, on stalks, yield the MHSs of the fibres.
We denote the weight filtrations by $W_jR^n {\pi}_* \QQ$ and $W_jR^n {\pi}_{!} \QQ$, the Hodge filtrations by $F^jR^n {\pi}_* \QQ$ and $F^jR^n {\pi}_{!} \QQ$, and the graded pieces by
$\operatorname{gr}_j^W R^n {\pi}_* \QQ$ and $\operatorname{gr}_j^W R^n {\pi}_! \QQ$.

We will consider the non-trivial cases $PR^{\kappa}{\pi}_* \QQ$ and $PR^{\kappa}{\pi}_{!} \QQ$. We will write $PR^{\kappa}{\pi}_* \QQ$ for the local system given by primitive cohomology and $PR^{\kappa}{\pi}_! \QQ $ for the local system given by primitive cohomology with compact support. 

Associated to the local system $R^n{\pi}_*\CC$ is a complex vector bundle with a flat connection, called the analytic Gauss--Manin connection, whose local system of flat sections is $R^k{\pi}_*\CC$ and local system of solutions is 
$R^k{\pi}_!\CC$ (\cite{Deligne-equations} or \cite[Section 9]{Voisin1}). 
Under the assumptions of Section \ref{subsec:relative-coho}, by \cite[Proposition 6.14]{Deligne-equations}, the analytic Gauss--Manin connection is the analytification of the algebraic Gauss-Manin connection. Throughout the paper, we will conflate the two objects and work in the analytic topology.
\smallskip

\section{Set up}
\label{sec:set-up}
In this section, we expand on the objects involved in Theorem \ref{thm:non-trivial-isomorphism} and prove some essential facts about them. 

We denote the rank of a local system $\mathbb{V}$ by $\rk{\mathbb{V}}$, and the order of a differential operator $D$ by $\mathrm{ord}\,D.$

\subsection{Hypergeometric local systems}
\label{sec:hgm-ls}
For more details on the material in this section, we refer the reader to \cite{BH, Katz, RV}.
\smallskip

For $n\ge 1$, let $\alpha:=(\alpha_1,\dots, \alpha_n)$ and $\beta:=(\beta_1,\dots,\beta_n)$ be multisets of rational numbers. To avoid redundancies, assume that there exists at least one pair $(\alpha_i, \beta_j)$ such that $\alpha_i -\beta_j \notin \ZZ$.
The \emph{hypergeometric differential operator} $H=H(\alpha, \beta)$ associated to $\alpha$ and $\beta$ is defined by\footnote{For convenience, we normalize the coordinate on $\Ps^1$ so that the singular point of $H$ over $\Gm$ is $t=1$.}  
\begin{equation}\label{eq:HG-diff-operator}
H(\alpha, \beta):=\prod_{i=1}^n (\theta+\beta_i -1) - t \prod_{i=1}^n (\theta+\alpha_i)
\end{equation}
where $t$ is a coordinate on $\PP^1$ and $\theta:= t\frac{d}{dt}.$
It is a differential operator of order $n$ on $\Ps^1$ with regular singularities at $0,1 $ and $\infty.$
The local system of solutions of the differential equation $H \cdot \varphi =0$, which we denote by $\HH$, is a complex local system of rank $n$ over $U=\PP^1 \setminus \{0,1,\infty\}$. We call $\HH$ a \emph{hypergeometric local system}.

\begin{rem}\label{rem:hypergeometric-functions}
If one of the $\beta_i$'s is $1$, the hypergeometric function 
\begin{equation}\label{eq:HGfunction-defi}
F\biggl(\begin{matrix} \alpha_1, \ \ldots, \ \alpha_n \\ \beta_1, \ \ldots, \ \beta_n \end{matrix} \:\bigg|\:\: t \biggr)=\sum_{k=0}^\infty \frac{(\alpha_1)_k(\alpha_2)_k\cdots (\alpha_n)_k}{(\be_1)_k(\be_2)_k\cdots (\be_n)_k} \: t^k
\end{equation} 
is a solution to $H(\alpha,\beta) \cdot \varphi =0$, where $(x)_k\coloneqq x(x+1)\cdots(x+k-1)$ for $k \geq 1$ and $(x)_0\coloneqq 1$. 
More generally,  
\[t^{1-\beta_i}
F\biggl(\begin{matrix} \alpha_1+1-\beta_i, \ \ldots, \ \alpha_n+1-\beta_i \\ \be_1+1-\beta_i, \ \ldots, \ \be_n+1-\beta_i \end{matrix} \:\bigg|\:\: t \biggr)
\quad \  i=1,\dots,n\]
are solutions of $H(\alpha,\beta) \cdot \varphi =0$; they form a basis of solutions near $t=0$ if the $\beta_i$'s are pairwise distinct modulo $\mathbb{Z}.$
\end{rem}

We fix a base point $t_0 \in U$. We denote by $\rho: \pi_1(U, t_0) \to \operatorname{GL}_n(\CC)$ the monodromy representation associated to $\HH$, which is defined up to conjugation. 
It is convenient to present $\pi_1(U, t_0)$ as generated by three loops $g_0, g_1,$ and $g_\infty$ around $0, 1,$ and $\infty$ respectively, satisfying the relation $g_\infty g_1 g_0=1.$ For $s= 0, 1, \infty$, we write $h_s \coloneqq \rho(g_s)$. 
We recall that:
\begin{enumerate}
    \item \label{item:char-pol}
    The characteristic polynomials of $h_\infty$ and $h_0^{-1}$ are 
\begin{equation}
\label{eq:char-polys}
q_\infty(T)=\prod_{j=1}^n (T-\exp{2\pi i \alpha_j})  \quad  \text{and} \quad q_0(T)=\prod_{j=1}^n (T-\exp{2\pi i \beta_j})
\end{equation} 
    \item \label{item:pseudo-reflec}The matrix $h_1$ is a pseudoreflection, i.e., $h_1-\mathrm{id}$ has rank one (where $\mathrm{id}$ denotes the identity matrix). 
\end{enumerate}
\smallskip

We now focus on irreducible hypergeometric local systems. 
It is well--known that a hypergeometric local system $\HH$ is irreducible if and only if $\alpha_i - \beta_j \notin \ZZ$ for all $i,j \in \{1, \dots, n\}$. Under this assumption, the isomorphism class of $\HH$ only depends on  $\alpha_i, \beta_j$ modulo $\ZZ$. This is implied by the following result:

\begin{prop}{\cite[Theorem 1.1]{Levelt}}
\label{pro:levelt} 
Assume $\alpha_i \neq \beta_j$ mod $\ZZ$ for all $i, j$.
Let $\rho \colon \pi_1(U, t_0) \to \GL(n,\CC)$ be any representation that satisfies \eqref{item:char-pol} and \eqref{item:pseudo-reflec}.
Let 
$
\rho_L \colon \pi_1(U, t_0) \to \GL(n,\CC)
$
be the representation determined by the assignment \[\rho_L(\gamma_\infty)\coloneqq M_\infty \qquad 
\rho_L(\gamma_0)\coloneqq M_0^{-1}\]
where $M_\infty$ and $M_0$ are the companion matrices of $q_\infty$ and $q_0$. 
Then, the two representations $\rho$ and $ \rho_L$ are isomorphic.
\end{prop}

The matrices $M_\infty$ and $M_0$ lie in $\mathrm{GL}(n, E)$, where $E$ is the subfield of $\CC$ generated by the coefficients of $q_\infty$ and $q_0$. One says that $\HH$ is defined over $E$.  
\begin{rem}
\label{rem:rigid}
A local system is called \emph{rigid} if its local monodromy operators determine its monodromy representation. 
By Proposition \ref{pro:levelt}, irreducible hypergeometric local systems are rigid.    
\end{rem}

A convenient format to encode an isomorphism class of irreducible hypergeometric local systems is the rational function $Q=q_\infty/ q_0 \in E(T)$, called the \emph{family parameter}, which records $\alpha$ and $\beta$ modulo $\ZZ$. 
\smallskip 

In this paper, we restrict to the case $E=\QQ$, in which case, $q_\infty$ and $q_0$ are products of cyclotomic polynomials.
Then $Q$ can be rewritten 
as:
\begin{equation}\label{eq:Q-gamma}
Q(T)= 
\frac{\prod_{\gamma_i<0} (T^{-\gamma_i}-1)}{\prod_{\gamma_i>0} (T^{\gamma_i}-1)}
\end{equation}
for a vector $\gamma:=(\gamma_1,\ldots,\gamma_l)$ of nonzero integers $\gamma_i$ that sum to zero. 
We call such a vector a \emph{gamma vector}, and say that $\gamma$ \emph{represents} $Q$. 
We say that $\gamma$ is \emph{reduced} if no pair of its entries sums to zero, and we say that $\gamma$ is \emph{prime} if its entries are relatively prime.  

Note that, up to reordering entries, there is a unique reduced gamma vector representing $Q$, 
and any other gamma vector representing $Q$ is obtained by appending pairs of opposite integers to it. The reduced gamma vector representing $Q$ is not necessarily prime. Still, one can always build a prime gamma vector that represents $Q$ by appending integers that are coprime with the greatest common divisor of the entries. 
\begin{example}
\label{exa:Q-gamma}
If $\alpha=(1/2, 1/2), \beta=(0,0)$,
then \[
Q=\frac{(T^2-1)^2}{(T-1)^4}
\] thus $\gamma=(-2,-2,1,1,1,1)$  is the reduced gamma vector representing $Q$; it is prime. 

If $\alpha=(1/4, 3/4), \beta=(0, 1/2)$,
then \[
Q=\frac{T^4-1}{(T^2-1)^2}
\] thus $\gamma=(-4,2,2)$ is the reduced gamma vector representing $Q$; it is not prime. Instead, $\gamma^m=(-4,2,2,-(2m+1),2m+1)$, $m\in \ZZ$, are prime gamma vectors representing $Q$. 
\end{example}

It is well-known that irreducible 
hypergeometric local systems admit a motivic interpretation \cite{Andre-G-functions, Katz}. An irreducible hypergeometric local system $\HH$ defined over $\QQ$ also supports a unique rational variation of Hodge structure (VHS) whose Hodge numbers are computable combinatorially in terms of $\alpha, \beta$ \cite{AG, Fedorov}. 
Our Theorem \ref{thm:non-trivial-isomorphism} yields explicit geometric realisations of this VHS. These are built out of prime gamma vectors representing $Q$ as described in the next section.

\begin{rem} 
In \cite{RV}, the definition of gamma vector includes the condition that $\gamma$ is reduced. 
In this paper, we remove this condition from the definition, as, unlike $\gamma$ being prime, it is not necessary for Theorem \ref{thm:non-trivial-isomorphism}. In fact, removing this condition allows us to address cases that would be otherwise excluded, such as the second local system in Example \ref{exa:Q-gamma}.  
\end{rem}

\begin{rem}
    We refer the reader to \cite[\S 1.3]{GGFRV} for details on the relation among the geometric realisations built out of different prime gamma vectors representing $Q$. 
\end{rem}

\begin{example}[Examples \ref{exa:curve-ft}, \ref{exa:curve-numbers}, \ref{exa:curve-cone}, \ref{exa:curve-u} continued]
\label{exa:curve-gamma}
Consider the multisets \[\alpha=\left(\frac{1}{4},\frac{1}{3}, \frac{2}{3}, \frac{3}{4}\right) \quad \beta=\left(\frac{1}{5}, \frac{2}{5}, \frac{3}{5},\frac{4}{5}\right)\]
and let $\HH$ be the corresponding local system. The vector $\gamma=(-5,-2,3,4)$ is the reduced gamma vector representing $Q$. The Hodge weight of the VHS supported by $\HH$ is one, and the Hodge numbers are $h^{1,0}=h^{0,1}=2$. 
Note that $\gamma$ is the generator of the lattice $R$ of Example \ref{exa:curve-u}.
\end{example}

\begin{example}[Examples \ref{exa:chebyshev-ft}, \ref{exa:chebyshev-hodge}, \ref{exa:chebyshev-cone}, \ref{exa:chebyshev-u} continued]
\label{exa:chebyshev-gamma}
Consider the multisets
\begin{align*}
  \alpha = \left(\frac{1}{30}, \frac{7}{30}, \frac{11}{30}, \frac{13}{30}, \frac{17}{30}, \frac{19}{30}, \frac{23}{30}, \frac{29}{30}\right) \quad 
  \beta  = \left(1,\frac{1}{2}, \frac{1}{3}, \frac{2}{3}, \frac{1}{5}, \frac{2}{5}, \frac{3}{5}, \frac{4}{5}\right)  
\end{align*}
and let $\HH$ be the corresponding local system. The vector $\gamma=(-30,-1,6,10,15)$ is the reduced gamma vector representing $Q$. The Hodge weight of the VHS supported by $\HH$ is two, and the Hodge numbers are $h^{2,0}=h^{0,2}=0, h^{1,1}=8$. Note that $\gamma$ is the generator of the lattice $R$ of Example \ref{exa:chebyshev-u}.
\end{example}

\subsection{The pair associated to a gamma vector}
\label{sec:pair} 
Some of the statements in Sections \ref{sec:toric-models} and \ref{sec:sing-quasiD} can be found in \cite[Section 3, Section 1.3]{RV, GGFRV}; for completeness, we supply their proofs.

Let $\gamma=(\gamma_1, \dots, \gamma_l)$ be a prime gamma vector. 

\begin{dfn}
Let $z_1, \dots, z_l$ be homogeneous coordinates  on $(\CC^\times)^{l-1}$. Write $z^\gamma\coloneqq \prod_{j=1}^lz_j^{\gamma_j}$ and $\Gamma \coloneqq \prod_{j=1}^l \gamma_j^{\gamma_j}$.
Let $(Z, \pi)$ be the pair defined by: 
\begin{equation}
\label{eq:Z-pi} Z\coloneqq\left(\sum_{j=1}^l z_j=0\right) \subset \TT^{l-1}
\quad \text{and} \quad \pi\coloneqq \frac{z^\gamma }{\Gamma } \colon Z \to \CC^\times
\end{equation}
We call $(Z, \pi)$ \emph{the pair associated to $\gamma$}.\footnote{We normalise the map $\pi$ so that $t=1 \in \CC^\times$ is the only critical value, see Proposition \ref{pro:(Z,pi)-prop}. This is consistent with our definition of hypergeometric differential operators.} 
\end{dfn}
\smallskip

The pair $(Z, \pi)$ determines a one-parameter family of affine varieties $Z_t\coloneqq \pi^{-1}(t)$, $t \in \CC^\times$, given by
\begin{equation}
   Z_t=\left( z_1+\dots + z_l=0, \, \Gamma t= z^\gamma  \right) \subset (\CC^\times)^{l-1}
\end{equation} 
We write $\kappa:=l-3$ for the dimension of $Z_t$.
The varieties $Z_t$  (up to scaling the parameter) are the varieties $Y_t$ in \cite[Equation (2)]{AG} when $\gamma$ has only one negative entry, and the varieties $V_\lambda$ in \cite[Equation (1.1)]{BCM} if $\Gamma t$ is invertible modulo $p$.  

\subsubsection{Toric models} 
\label{sec:toric-models}
It is convenient to consider different presentations for $(Z, \pi)$, which, following \cite[$\S 4$]{RV}, we call  \emph{toric models}. We briefly recall the construction below.

Write $d \coloneqq l-2$, and pick an $l \times l$ integral matrix 
\begin{equation}\label{eq:matrix-gamma}
A= \begin{bmatrix}
1 & 1 & \cdots & 1 \\
m_{11} & m_{12} & \cdots & m_{1l} \\
\vdots & \vdots & \ddots & \vdots \\
m_{d1} & m_{d2} & \cdots & m_{dl} \\
k_1 & k_2 & \cdots & k_l
\end{bmatrix}
\end{equation}
such that 
\begin{enumerate}
    \item the first $l-1$ rows generate the kernel of the map $\gamma \colon \ZZ^l \to \ZZ$ given by $x \mapsto \gamma \cdot x$; 
    \item the last row satisfies $\sum_{j=1}^{l} \gamma_j k_j=1$.
\end{enumerate} 
Such a matrix $A$ exists since the entries $\gamma$ sum to zero and $\gamma$ is prime. Note that $\det A=\pm 1$. 
\smallskip

Denote by $m_j$, $j=1, \dots, l$, the columns of the $d \times l$ submatrix of ${A}$ with entries $m_{ij}$. Let $\TT^d \simeq (\CC^\times)^d$ be a torus with coordinates $x_1, \dots, x_d$, and let $t$ be a coordinate on $\CC^\times$.

\begin{lem}
\label{lem:change-of-variable}
The substitution 
   \[z_j \mapsto \gamma_j t^{k_j} x^{m_j}, \quad
j=1, \dots, l\] 
defines an isomorphism \mbox{$\TT^d \times \CC^\times \to (\CC^\times)^{l-1}$}. 
\end{lem}

\begin{proof}
For $h \in \{1, \dots, l\}$, let $A_h$ be the $l \times l$ integral matrix with columns

\[
\label{eq:A_h}
    \begin{pmatrix}
1  \\
m_{j}-m_h\\
k_j -k_h
\end{pmatrix} \quad 
     j=1, \dots, l
\]
\vspace{0.01cm}

\noindent The matrix $A_h$ satisfies the conditions (1) and (2) above. Indeed, 
    \[ \sum_{j=1}^l \gamma_j (m_j-m_h)= \sum_{j=1}^l \gamma_j m_j-m_h \cdot \sum_{j=1}^l\gamma_j=0
    \]
    Moreover, the first $l-1$ rows of $A_h$ span $\ker(\gamma)$ as, for all $i=1, \dots, d$, 
    \[
    (m_{i1}, \dots, m_{il})=(m_{i1}-m_{ih}, \dots, m_{il}-m_{ih})+m_{ih}\cdot (1, \dots, 1)
    \]
    Finally, 
    \[
    \sum_{j=1}^l \gamma_j (k_j-k_h)= \sum_{j=1}^l \gamma_j k_j -k_h \cdot \sum \gamma_j=1
    \]
We have in particular that $\det A_h=\pm1$. 

Let $A^{'}_h$ be the $(l-1) \times (l-1)$ submatrix of $A_h$ obtained by removing the first row and the $h$th column of $A_h$. Since the $h$th column of $A_h$ has first entry equal to $1$ and all other entries equal to $0$, we have that $\det A_h= \pm \det A^{\prime}_h$. 
It follows that $\det A^{\prime}_h=\pm 1$.
This is equivalent to the fact that the columns of $A^{\prime}_h$, that is, the vectors 
\[
\begin{pmatrix}
m_{j}-m_h\\
k_j -k_h
\end{pmatrix} \quad 
     j=1, \dots, l, j \neq h
\] span $\ZZ^{l-1}$. This proves the statement. 
\end{proof}

In the coordinates $x_1, \dots, x_d,  t$, the pair $(Z, \pi)$ takes the form: 
\begin{equation}
\label{eq:toric-model}
   Z= \left(  \sum_{j=1}^l \gamma_j t^{k_j} x^{m_j}=0
   \right)
\subset \TT^d \times \CC^\times  \qquad \pi=\mathrm{pr}_t \colon Z \to \CC^\times 
\end{equation}
where 
$\mathrm{pr}_t$ denotes the restriction to $Z$ of the projection $\TT^d \times \CC^\times \to \CC^\times$ to $t$. Equation \eqref{eq:toric-model} is called a \emph{toric model} of $(Z, \pi)$. 

In what follows, we denote by $f \in \CC[x_1^{\pm 1}, \dots, x_d^{\pm 1}, t^{\pm 1}]$ the Laurent polynomial defining  $Z$ in \eqref{eq:toric-model}, and, for $t \in \CC^\times$ fixed, we write $f_t \coloneqq f(\cdot,t) \in \CC[x_1^{\pm 1}, \dots, x_d^{\pm 1}]$.
Moreover, we denote by $\Delta \subset \RR^d$ the convex hull of the vectors $m_j$, that is, the Newton polytope of the Laurent polynomials  $f_t$.


\smallskip

We conclude by emphasizing a few properties of $A$ and $\Delta$ and by discussing some examples.

\begin{rem}
\label{rem:A-matrix}
Equation \eqref{eq:Z-pi} defines $(Z, \pi)$ without any choice, while Equation \eqref{eq:toric-model} depends on the choice of the matrix ${A}$. 
This matrix is unique up to left multiplication by matrices of the form:
\begin{displaymath}
\left( \begin{array}{ccc}
1 & 0 & 0 \\
a & B & 0\\
c & d & 1
\end{array} \right)
\end{displaymath}
where $a\in \ZZ^{d \times 1}$, $B \in \mathrm{GL}(d, \ZZ)$, $c\in \ZZ$, $d \in \ZZ^{1 \times d}$. In particular, the polytope $\Delta \subset \RR^d$ is well-defined up to affine linear transformations.

For example, given $A$, the matrix $A_h$ considered in the proof of Lemma \ref{lem:change-of-variable} is obtained by the choices $a=-m_h, B=\mathrm{Id}$, $c=-k_h$, $d=0$.
\end{rem}

\begin{rem}
\label{rem:gamma-relations}
The vector $\gamma$ spans the lattice of integral relations among the vectors 
    \begin{equation}
        \label{eq:KT}
    \begin{pmatrix}
1  \\
m_{j} 
\end{pmatrix} \in \ZZ^{l-1}\quad 
     j=1, \dots, l
     \end{equation}

This implies that the vectors $m_j$ are all distinct. Equivalently, for all $t\in \CC^\times$,  $f_t$ is the sum of $l$ distinct monomial terms. 
\end{rem}

\begin{rem}
\label{rem:kernel}
The first $l-1$ rows of $\mathcal{A}$ span a saturated sublattice of $\ZZ^l$ of rank $l-1$. 
Equivalently, the maximal minors of the $(l-1)\times l$ submatrix of $A$ defined by the first $l-1$ rows are relatively prime. In particular, the vectors 
\eqref{eq:KT}
    span $\ZZ^{l-1}$. This is easily seen to be equivalent to the fact that, for any $k \in \{1, \dots, l\}$, the vectors $m_j-m_k$, $j=1, \dots, l$, span $\ZZ^d$. 
\end{rem}

\begin{rem}
\label{rem:mj-polytope}
The polytope $\Delta$ is full-dimensional and simplicial\footnote{A polytope is called simplicial if all its proper faces are simplices.}. 
Indeed, since for any fixed $k$ the vectors $m_j-m_{k}$ generate $\ZZ^d$, $\Delta$ has dimension $d$. Then, either $\Delta$ has $d+1$ vertices, i.e., it is a simplex, or it has $d+2$ vertices. In the second case, any facet of $\Delta$ must be a simplex: if a facet $F \subset \Delta$ had $d+1$ vertices, say $m_2, \dots, m_l$, then  for all $j \neq 1$,  $\langle m_j, u\rangle =b$ for some $u \in \ZZ^d, b \in \ZZ$, while $\langle m_1, u\rangle =a>b$; on the other hand, 
\[
0=\left\langle \sum_{j=1}^l \gamma_j m_j, u \right\rangle =\sum_{j=1}^l \gamma_j \langle m_j, u \rangle=  \gamma_1a+\sum_{j=2}^l \gamma_j b   = \ \gamma_1 (a -b)
\] 
implies that $a=b$ since $\ga_1 \neq 0$. 

It is simple to see that $\Delta$ is a simplex if and only if either $\gamma$ or $-\gamma\coloneqq (-\gamma_1, \dots, -\gamma_l)$  has only one negative entry. More precisely, if $\gamma_k$ is the only negative/positive entry of $\gamma$, then $m_k$ lies in the interior of $\Delta$. 
\end{rem}

\begin{example}[Examples \ref{exa:curve-ft}, \ref{exa:curve-numbers}, \ref{exa:curve-cone}, \ref{exa:curve-u}, \ref{exa:curve-gamma} continued]
\label{eq:curve-pair}
Consider the gamma vector $\gamma= (-5,-2,3,4)$. One can pick the matrix 
\begin{equation*}
   A= \begin{bmatrix}
        1 & 1 & 1 & 1 \\
        2 & 0 & 2 & 1 \\
        0 & 3 & 2 & 0 \\
        0 & 1 &  1 & 0
    \end{bmatrix}.
\end{equation*}
Then, one obtains the Laurent polynomial
    \[
    f=-5x_1^2 -2t  x_2^3+ 3tx_1^2x_2^2+4x_1
    \] 
Then the pair $(Z, \pi)$ is the one-parameter family of curves of Example \ref{exa:curve-ft}. 
\end{example}
\begin{example}[Examples \ref{exa:chebyshev-ft}, \ref{exa:chebyshev-hodge}, \ref{exa:chebyshev-cone}, \ref{exa:chebyshev-u}, \ref{exa:chebyshev-gamma} continued] 
Consider the gamma vector $\gamma= (-30,-1,6,10,15)$.  One can pick the matrix 
\begin{equation*}
   A= \begin{bmatrix}
        1 & \textcolor{white}{-}1 & \textcolor{white}{-}1 & \textcolor{white}{-}1 & \textcolor{white}{-}1\\
        1 & \textcolor{white}{-}0 & \textcolor{white}{-} 5 & \textcolor{white}{-}0 & \textcolor{white}{-} 0\\
        1 & \textcolor{white}{-}0 & \textcolor{white}{-} 0 & \textcolor{white}{-} 3 & \textcolor{white}{-} 0\\
        1 & \textcolor{white}{-}0 & \textcolor{white}{-}0 & \textcolor{white}{-}0 & \textcolor{white}{-}2\\
        0 & -1 &  \textcolor{white}{-}0 & \textcolor{white}{-}0 &\textcolor{white}{-}0
    \end{bmatrix}.
\end{equation*}
Then one obtains the Laurent polynomial
\[
f=-30 x_1x_2x_3 -\frac{1}{t}+6x_1^5+10 x_2^3 +15 x_3^2
\]
The pair $(Z, \pi)$ is the one-parameter family of surfaces of Example \ref{exa:chebyshev-ft}. The vector $\gamma$ and the family $(Z, \pi)$ have been studied in \cite{FRV-duke}.
\end{example}
\smallskip

\subsubsection{Singularities and quasi-$\Delta$-regularity.} 
\label{sec:sing-quasiD}

By choosing a toric model \eqref{eq:toric-model}, we may view the pair $(Z, \pi)$ as a one-parameter family of affine hypersurfaces $(Z_t)_{t \in \CC^\times}$ in $\TT^d$. 
We now discuss the singularities of the family and the quasi-$\Delta$ regularity of $Z_t$.  

\begin{prop}
\label{pro:(Z,pi)-prop}
The following statements hold:
    \begin{enumerate}
    \item For all $t \neq 1$, $Z_t$ is nonsingular. 
    \item   $Z_1$ has a unique ODP. 
    \item  For all $t \in \CC^\times$, $Z_t$ is quasi--$\Delta$--regular. In particular, for all $t \neq 1$, $Z_t$ is $\Delta$--regular. 
    \item $Z$ is smooth. 
    \end{enumerate}
\end{prop} 

Note that Proposition \eqref{pro:(Z,pi)-prop} implies that $\pi \colon Z \to \CC^\times$ (as well as the restriction of $\pi$ to the preimage of any affine open subset of $\CC^\times$) satisfies the assumptions of Theorem \ref{thm:compactifying-delta}. Moreover, the restriction of $\pi$ over $U\coloneqq \CC^\times \setminus \{1\}$ is smooth, so the results of Sections \ref{subsec:relative-coho} and \ref{sec:local-systems} apply to this map. 
\smallskip

In what follows, we prove a slightly more general result, from which Proposition \ref{pro:(Z,pi)-prop} follows almost immediately. The symbols $\Gamma$, ${A}$,  $m_j$, $j=1, \dots, l$, and $\Delta$ have the same meanings as in Section \ref{sec:toric-models}.

\begin{prop}
\label{prop:u-result} 

Let $\mathcal{F}=\sum_{j=1}^l u_j x^{m_j} \in \CC[x_1^{\pm 1}, \dots, x_d^{\pm 1}, u_1^{\pm 1}, \dots, u_1^{\pm 1}]$. Let $\mathcal{Z}=Z_\mathcal{F}$ be the zero locus of $\mathcal{F}$ in $\TT^d \times (\CC^\times)^l$. 
For $u=(u_1, \dots, u_l) \in (\CC^\times)^l$ fixed, let $\mathcal{F}_u=\mathcal{F}(\cdot,u)$, and let $\mathcal{Z}_u=Z_{\mathcal{F}_u}$ be the zero of $\mathcal{F}_u$ in $\TT^d$. Then:   
    \begin{enumerate}
    \item  If $\prod_{j=1}^l u_j^{\gamma_j} \neq \Gamma$, then $\mathcal{Z}_u$ is smooth. 
    \item  If $\prod_{j=1}^l u_j^{\gamma_j} = \Gamma$, then $\mathcal{Z}_u$ has a unique ODP.  
    \item  For all $u \in (\CC^\times)^l$, $\mathcal{Z}_u$ is quasi $\Delta$-regular. In particular, if $\prod_{j=1}^l u_j^{\gamma_j} \neq \Gamma$, then $\mathcal{Z}_u$ is $\Delta$-regular. 
    \item $\mathcal{Z}$ is smooth.
    \end{enumerate}
\end{prop}

\begin{proof}[Proof of Proposition \ref{prop:u-result}] 
Fix $u \in (\CC^\times)^l$.
For $i=1, \dots, d$, we have
\begin{equation}
 x_i \frac{\partial  \mathcal{F}_u }{\partial x_i} = \sum_{j=1}^l m_{ij} u_j x^{m_j}
\end{equation}
The variety $\mathcal{Z}_u$ is singular if and only there is a point  $x=(x_1, \dots, x_d) \in \TT^d$ that satisfies the equations
\begin{equation}
\label{eq:singularity}
\sum_{j=1}^l u_j x^{m_j}=0,  \quad 
  \sum_{j=1}^l m_{ij}  u_j x^{m_{j}}=0 \ \text{ for } i =1,\dots,d 
\end{equation} 
which we can rewrite as
\begin{equation*}
\label{eq:1m-system}
\begin{pmatrix}
1 & \dots & 1\\
m_1 & \dots & m_l
\end{pmatrix} \begin{pmatrix}
u_1 x^{m_1}\\
\vdots\\
u_l x^{m_l}
\end{pmatrix} = 0
\end{equation*}
Then, by Remark \ref{rem:gamma-relations}, $x \in \TT^d$ satisfies \eqref{eq:singularity} if and only if 
    \begin{equation}
    \label{eq:y-gamma}
    (u_1 x^{m_1}, \dots, u_l x^{m_l}) =(c \gamma_1,  \dots, c \gamma_l)\end{equation} for some constant $c \in \CC^\times$. 
If such $x$ exists, then 
\[
\prod_{j=1}^l u_j^{\gamma_j}= \prod_{j=1}^l \left( u_j x^{m_j} \right)^{\gamma_j}= \prod_{j=1}^l\left(  c\gamma_j\right)^{\gamma_j}= c^{\sum_{j=0}^l \gamma_j} \cdot \Gamma=\Gamma
\]
which proves $(1)$. 

Assume now that $u \in (\CC^\times)^l$ is such that $\prod_{j=1}^l u_j^{\gamma_j}=\Gamma$. 
Then there is a unique $x \in \TT^d$ satisfying \eqref{eq:y-gamma}. 
Indeed, if \eqref{eq:y-gamma} holds, then, for $j=1, \dots, l$, 
\[ x^{m_j-m_1}=\frac{u_1 \gamma_j}{\gamma_1 u_j}
\]
Since the vectors $m_j-m_1$, $j=2, \dots, l$ generate generate $\ZZ^d$ (see Remark \ref{rem:kernel}), these identities determine $x=(x_1, \dots, x_d)$ uniquely (note that $c=u_1x^{m_1}/\gamma_1$ is also uniquely determined).
This shows that, if $\prod_{j=1}^l u_j^{\gamma_j}=\Gamma$, then $\mathcal{Z}_u$ has a unique singular point. To show $(2)$, it remains to show that the singular point is an ODP. 

Up to changing the variables $x_i$ by some constants, we may assume that $u=\gamma$. Then the singular point of $\mathcal{Z}_u$ is $x=(1, \dots, 1) \in \TT^d$. The Hessian of $\mathcal{F}_u$ at $x=(1, \dots, 1)$ is the matrix $H$ with entries 
\begin{equation}\label{eq:hessian}
H_{i,j}= \sum_{h=1}^l \gamma_h m_{i,h} m_{j,h}  x^{m_h}\bigg|_{x=1} = \sum_{h=1}^l \gamma_h m_{i,h} m_{j,h}.
\end{equation}
Equivalently, letting $M$ be the $d \times l$ submatrix of ${A}$ with entries $m_{ij}$ and $\mathrm{diag}(\gamma)$  be the diagonal matrix with diagonal entries $\gamma_1, \dots, \gamma_l$, we have that $H=M \cdot \mathrm{diag}(\gamma) \cdot M^T$.
Moreover, 
the product ${A} \cdot \mathrm{diag}(\gamma)  \cdot {A}^T$ is of the form: 
\begin{equation}
\label{eq:prod-H}
{A} \cdot \mathrm{diag}(\gamma)  \cdot {A}^T =\begin{bmatrix}
0 & 0 & 1 \\
0 & H & * \\
1 & * & *
\end{bmatrix} 
\end{equation}
Since $\det({A})=\pm 1$, one has $\det({A} \cdot \mathrm{diag}(\gamma)  \cdot {A}^T)=\prod_{j=1}^l \gamma_j$. Combining this with \eqref{eq:prod-H} yields  $\prod_{j=1}^l \gamma_j=-\det(H)$,  thus $H$ is invertible. This shows that $x$ is an ODP.

To prove $(3)$, we verify that, for each face $F$ of $\Delta$ of dimension $r \in (0,d)$, the zero locus of ${\mathcal{F}_u}_{|F}$ in $\TT^d$ is smooth. 
For any face $F\subset \Delta$, the zero locus of ${\mathcal{F}_u}_{|F}$ in $\TT^d$ is singular if and only if there exists $x \in \TT^d$ such that 
\begin{equation}
\label{eq:singularity-face}
   \sum_{m_j \in F } u_j x^{m_{j}}=0, \quad  
    \sum_{m_j \in F }   m_{ij} u_j x^{m_j}=0  \ \text{ for } i=1,\dots,d
\end{equation}
Let $F\subset \Delta$ be a face of dimension $r$. Then $F$ is a $r$-simplex, i.e., it has $r+1$ vertices $m_{j_1}, \dots, m_{j_{r+1}}$, and does not contain any other vertex $m_j$ (see Remark \ref{rem:mj-polytope}). 
Then, we can rewrite  
\eqref{eq:singularity-face} as 
\begin{equation}
\label{eq:ls-face}
\begin{pmatrix}
1 & \cdots  & 1\\
m_{j_1} & \cdots & \ m_{j_{r+1}}
\end{pmatrix}\begin{pmatrix}
u_1 x^{m_1}\\
\vdots\\
u_l x^{m_l}
\end{pmatrix} = 0
\end{equation}
Since any maximal minor of the $(l-1)\times l$ submatrix of ${A}$ defined by its first $l-1$ rows is nonzero (see Remark \ref{rem:kernel}) and $r+1<l-1$,  \eqref{eq:ls-face} does not admit any non-trivial solution.  

Statement $(4)$ follows immediately from the fact that, for all $h=1, \dots, l$,
\[u_h \frac{\partial \mathcal{F}} {\partial u_h}= u_h x^{m_h}\]
and that, for each $k$ fixed, the vectors $m_j-m_k$ generate $\ZZ^d$.

\end{proof}

\begin{proof}[Proof of Proposition \ref{pro:(Z,pi)-prop}]
Statements $(1), (2), (3)$ follow immediately from the corresponding statements in Proposition \ref{prop:u-result} by setting $u_j=\gamma_j t^{k_j}$ for all $j=1, \dots, l$. 
To prove $(4)$, one can use the same technique of the proof of Proposition \ref{prop:u-result}$(3)$, relying now on the fact that the matrix ${A}$ is invertible.
\end{proof}

\begin{rem}
In view of Examples \ref{exa:curve-gamma} and \ref{exa:chebyshev-gamma}, the (quasi-)$\Delta$-regularity properties of the curves of Example \ref{exa:curve-ft} and the surfaces of Example \ref{exa:chebyshev-ft} follow directly from Proposition \ref{pro:(Z,pi)-prop}. In the same way, the locus of $\Delta$-regularity of Example \ref{exa:curve-u} is described by
Proposition \ref{prop:u-result}. 
\end{rem}

\subsubsection{Irreducibility}
In this section, we prove the following result:

\begin{prop}\label{prop:irreduciblity}
   If $d \geq 2$, then, for all $t \in \CC^\times$, $Z_t$ is irreducible.
\end{prop}
Note that requiring that $d\geq 2$ is equivalent to requiring that $Z_t$ has dimension $\kappa=d-1 \geq 1$. Moreover, the statement is clearly false when $d=1$. For instance, if $\ga=(-2,1,1)$, then for any $t \in \CC^\times$, $Z_t$ is given by two points (counted with multiplicity).

Proposition \ref{prop:irreduciblity} is crucial in the proof of Theorem \ref{thm:curves-even}. We will only use it when the varieties $Z_t$ are curves; nevertheless, for completeness, we prove it in full generality.

We view $Z_t \subset \TT^d$ as the zero locus of the Laurent polynomial $f_t=\sum_{j=1}^l \gamma_j t^{k_j}x^{m_j}$. Recall that the Newton polytope $\Delta \subset \RR^d$ of $f_t$ is $d$-dimensional and simplicial (see Remark \ref{rem:mj-polytope}). In particular, when $d=2$, $\Delta$ is either a triangle or a quadrilateral.
\smallskip

We first prove the proposition when $d>2$ or $\Delta$ is a triangle. 

\begin{proof}[Proof of Proposition \ref{prop:irreduciblity} for $d>2$ or $\Delta$ a triangle.]
Let $t \in \CC^\times$. We may assume that $f_t$ is a polynomial (see Remark \ref{rem:A-matrix}).
Suppose, for contradiction, that $f_t= gh$ for two polynomials $g,h$ which are not units. 
Then, by \cite{Ostro}, $\Delta$ is the Minkowski sum of the Newton polytopes of $g$, $h$:
\begin{equation}
\label{eq:sum}
\Delta= \operatorname{Newt}(g) +\operatorname{Newt}(h) 
\end{equation} 
We claim that this is impossible.
Since $\Delta$ is simplicial, all its two-dimensional faces are triangles. By \cite[(13)]{Decomp-convex}, if a decomposition of $\Delta$ as in \eqref{eq:sum} exists, then $\operatorname{Newt}(g) = r\operatorname{Newt}(h)$ for some integer $r>0$.
Thus, $\Delta = kP$ for some integer $k>1$ and some lattice polytope $P$. Then, for $j=1, \dots, l$, the vectors $m_j$ satisfy $m_j = k n_j$ for some $n_j \in \Z^{d}$. This contradicts the fact that $m_j-m_1$, $j=2,\dots,l$, span $\Z^{d}$ (see Remark \ref{rem:kernel}).
\end{proof}

It remains to prove the proposition when $\Delta$ is a quadrilateral. This happens if and only if $\gamma\in \ZZ^4$ has two negative entries and two positive entries (see Remark \ref{rem:mj-polytope}).
In this situation, the argument above does not work since $\Delta$ may be a (non-trivial) Minkowski sum:

\begin{example}
Consider the gamma vector $\ga= (-18, 20, -5, 3)$. We may choose the $m_j$ as the columns of the matrix
\[
M= \begin{bmatrix}
1 & 1 & 4 & 6 \\
0 & 1 & 4 & 0 \\
\end{bmatrix}.
\]
Then, $\Delta$ is the Minkowski sum of two  polytopes defined as the convex hull of the columns of two matrices 
\[
\begin{bmatrix}
1 & 1 & 2 & 3 \\
0 & 1 & 2 & 0
\end{bmatrix} \qquad
\begin{bmatrix}
0 & 2 & 3 \\
0 & 2 & 0
\end{bmatrix}
\]
\end{example}
\smallskip

\begin{proof}[Proof of Proposition \ref{prop:irreduciblity} for $\Delta$ a quadrilateral]
Let $t \in \CC^\times$. For $j=1,2,3,4$, set $u_j=\gamma_j t^{k_j}$. By Lemma \ref{lem:Z-is-covered} and Lemma \ref{lem:W-is-irreducible}, $Z_t=\mathcal{Z}_u$ admits an étale covering $W_t \to Z_t$, where $W_t \coloneqq \mathcal{W}_u \subset \TT^2$ is an irreducible hypersurface. It follows that $Z_t$ is irreducible.
\end{proof}

\begin{lem}\label{lem:Z-is-covered}
For $u=(u_1, u_2, u_3,u_4)\in (\CC^\times)^4$ fixed, let $\mathcal{Z}_u=(\sum_{j=1}^4 u_j x^{m_j}=0) \subset \TT^2$. 
There is an étale covering $\mathcal{W}_u \to \mathcal{Z}_{u}$, where $\mathcal{W}_u \subset \TT^2$ is the zero locus of a polynomial $h$ of the form \[h(y_1,y_2)=h_1(y_1)+h_2(y_2)
\]with $y_1,y_2$ coordinates on $\TT^2$ and
\[
h_1(y_1)= v_1 y_1^{a_1}+v_2y_1^{a_2} \qquad h_2(y_2)=v_3y_2^{a_3}+v_4y_2^{a_4} 
\]
such that $\gcd(a_1,a_2)=\gcd(a_3,a_4)=1$ and $\deg(h_1) \neq \deg(h_2)$.
\end{lem}

\begin{proof}
Since $\gamma$ is a gamma vector of length $4$, it must be that $\gamma_i+\gamma_j \neq 0$ for all $i,j$.  
Moreover, up to reordering the entries of $\gamma$ and replacing $\gamma$ by $-\gamma$, we can assume that  $|\gamma_1| <\gamma_3 \le \gamma_4 < |\gamma_2|$.

For each $i,j$, let $d_{ij}= \gcd(\gamma_i,\gamma_j)$.
Set 
\begin{align*}
a= -\ga_1 / d_{14} \quad    A = \ga_4/d_{14} \quad   b= \ga_3/ d_{23} \quad 
     B = -\ga_2/ d_{23} 
\end{align*}
Then, $a,A,b,B$ are positive, $\gcd(a,A)= \gcd(b,B)=1$ and $a<A$, $b<B$.

Note that 
$
d_{14}(-a+A)=d_{23}(B-b)$. Since $\gamma$ is prime, this implies in particular that $d_{14}|(B-b)$ and $d_{23}|(-a+A)$. Therefore, 
$
c \coloneqq (-a+A)/d_{23}=(B-b)/d_{14} \in \NN
$. 

Now consider the matrix
\[
\begin{bmatrix}
    A & 0 &0 &a \\
    0 & b & B &0
\end{bmatrix}
\]
Its rows, together with the vector $(1,1,1,1)$, span a finite index sublattice of the kernel of $\gamma \colon \ZZ^4 \to \ZZ$. Then, by \cite[Proposition 3.6]{asem-hgms-toric}, 
the hypersurface of $\TT^2$ given by the zero locus of the polynomial
\begin{equation}
\label{eq:candidate}
u_1  y_1^{A}+ u_2 y_2^{b}  +u_3 y_2^{B} +u_4 y_1^{a}
\end{equation}
where $y_1, y_2$ are coordinates on $\TT^2$,
is an étale covering of $\mathcal{Z}_{u}$. 
If $A\neq B$, 
we denote by $h$ the polynomial in \eqref{eq:candidate}, and the statement is proved.

If $A=B$, we write $n \coloneqq A=B$. Then we have that 
\[
\gamma=\frac{1}{c}\left(-a(n-b), -n(-a+n), b(-a+n), n(n-b)\right)
\]
Since $|\gamma_1| < \gamma_3$, we have that $a(n-b)< b(n-a)$. In particular, $a < b$.  Moreover, since $\gcd(a,n)=\gcd(b,n)=1$, we have that
\[
\gcd(a(n-b), b(n-a))= \gcd(n-a,n-b) \gcd(a,b) 
\] We write $d\coloneqq\gcd(n-a,n-b), e\coloneqq \gcd(a,b)$.

Consider the matrix 
\[
\begin{bmatrix}
    \frac{b(n-a)}{de} & 0 &\frac{a(n-b)}{de} &0 \\
    0 & \frac{n-b}{d}& 0 & \frac{n-a}{d}
\end{bmatrix}
\]
Its rows, together with $(1,1,1,1)$, span a finite index sublattice of the kernel of $\gamma$. Again, by \cite[Proposition 3.6]{asem-hgms-toric}, 
the hypersurface of $\TT^2$ given by the zero locus of the polynomial
\begin{equation}
\label{eq:candidateprime}
u_1  y_1^{\frac{b(n-a)}{de}}+ u_2 y_2^{\frac{n-b}{d}}  +u_3 y_1^{\frac{a(n-b)}{de}} +u_4 y_2^{\frac{n-a}{d}}
\end{equation}
where $y_1, y_2$ are coordinates on $\TT^2$,
is an étale covering of $\mathcal{Z}_{u}$. Note that, by construction,  
\[\gcd\left(\frac{b(n-a)}{de}, \frac{a(n-b)}{de}\right)=1 \quad \text{and} \quad  \gcd\left(\frac{n-b}{d}, \frac{n-a}{d}\right)=1\]
Moreover, $\frac{b(n-a)}{de} \neq \frac{n-a}{d}$ as otherwise $b=e=\gcd(a,b)$ which contradicts the fact that $a<b$.
Then, denoting by $h$ the polynomial in \eqref{eq:candidateprime}, the statement is proved. 
\end{proof}

\begin{example}
We give two examples to illustrate the last proof. 
Consider the gamma vector $\ga=(-6,-1,2,5)$. The curve $\mathcal{Z}_u$ given by
\[
u_1 +u_2 x_1^{5}x_2^2 +u_3 x_2 + u_4 x_1=0\]
admits an étale covering by the curve $\mathcal{W}_u$ given by
\[
u_1 y_2+u_2 y_1^5+u_3 y_2^3 +u_4 y_1=0
\]
Explicitly, the covering map is $2$-to-$1$, given by 
    $
    x_1 \mapsto y_1/y_2, x_2 \mapsto y_2^2
    $. 

Now consider the gamma vector $\ga=(-6,-1,3,4)$. The curve $\mathcal{Z}_u$ given by
    \[
    u_1x_2^2 +u_2x_1^{3} +u_3x_1 + u_4x_2^3=0
    \]
admits an étale covering by the curve $\mathcal{W}_u$ given by
    \[
   u_1 y_2 +u_2 y_1^4 +u_3 y_2^2+u_4 y_1=0
    \]
The covering map is an isomorphism, given by
    $
    x_1 \mapsto y_1^2/y_2, x_2 \mapsto y_1/y_2
    $.
\end{example}

\begin{lem}\label{lem:W-is-irreducible}
    The hypersurface $\mathcal{W}_u$ in Lemma \ref{lem:Z-is-covered} is irreducible.
\end{lem}
To prove Lemma \ref{lem:W-is-irreducible}, we need to introduce some notation. 
A polynomial $f(x) \in \CC[x]$ is called \emph{indecomposable} if whenever $f(x) = g(h(x))$ for $g,h\in \CC[x]$, then $h(x)= ax+b$ or $g(x) =ax+b$ for some $a,b \in \CC$.
\begin{proof}
Lemma \ref{lem:indecomp} below implies that the polynomials $h_1,h_2$ in Lemma  \ref{lem:Z-is-covered} are indecomposable. Furthermore, the degrees of $h_1$ and $h_2$ are different and nonzero.
Thus, by \cite[Theorem 1.1]{consequences-of-classific}, the polynomial $h=h_1+h_2$ is irreducible.\footnote{ 
We note that Theorem 1.1 in \cite{consequences-of-classific} is unconditional since the classification of finite simple groups is now a theorem.}
\end{proof}

\begin{lem}\label{lem:indecomp}
Let $m>n >0$ be positive integers with $\gcd(m,n)=1$. The polynomial 
$
x^{m}+x^{n}
$
is indecomposable. 
\end{lem}
\begin{proof}
Suppose for contradiction that $x^{m}+x^{n}=g(h(x)) $ for some polynomials $g,h \in \CC[x]$ of degrees strictly greater than $1$.
Write
\begin{align*}
    h(x) &= x^{k} P(x) + h(0) \qquad \text{ with } P(0) \neq 0 \\
    g(x) &= \left(x-h(0) \right)^{d} Q(x)  \quad \text{ with } Q(h(0)) \neq 0 , d \ge 1
\end{align*}
Then,
\[
g(h(x)) = \left(x^{k}P(x)\right)^{d} Q(h(x)) = x^{kd} P(x)^{d} Q(h(x)) 
\]
Since $P(0)^{d}Q(h(0)) \neq 0$, it must be that $x^{kd} = x^{n}$.
Thus,
\begin{equation}    \label{eq:indecomposable}
P(x)^{d} Q(h(x)) = x^{m-n} +1 
\end{equation}
We distinguish two cases.

If $P(x) = c \in \CC$, then $h(x) = c x^{k} + h(0)$ with $k \ge 2$, and 
$
x^{m} + x^{n} = g(h(x)) = F(x^{k})
$
for some $F \in \CC[x]$. Hence $k$ divides both $m$ and $n$, which contradicts the fact that $\gcd(m,n) = 1$.

If $P(x)$ is not constant, then, since the right-hand side of \eqref{eq:indecomposable} has only simple roots, one has that $d = 1$, and thus $k = n$.
Equation \eqref{eq:indecomposable} becomes
\begin{equation}\label{eq:indecomposable-2}
    P(x) Q\left(x^{n} P(x) + h(0)\right) = x^{m-n} + 1
\end{equation}
Let $P(x) = A \prod_{i=1}^r (x - z_i)$. Note that for each $i$, $z_i$ is also a root of $x^{m-n}+1$. Plugging $x = z_i$, for $i=1,\dots,r$, into \eqref{eq:indecomposable-2} gives
\begin{equation}
\label{eq:indecomposable-3}
\begin{split}
   Q(h(0))&= \lim_{x \to z_i } \frac{x^{m-n} + 1}{x - z_i} \cdot \frac{x - z_i}{P(x)} 
   = (m-n) z_i^{m-n-1} \cdot \frac{1}{P'(z_i)} \\ &= - (m-n) z_i^{-1} \cdot \frac{1}{P'(z_i)} = -\frac{m-n}{z_i P'(z_i)}
\end{split}
\end{equation}
Setting $x = 0$ in \eqref{eq:indecomposable-2} yields
\begin{equation}\label{eq:indecomposable-4}
    P(0) Q(h(0)) = 1
\end{equation}
Now consider the polynomial
\[
W(x) = Q(h(0))\, x P'(x) + m - n - (m-n) Q(h(0)) P(x)
\]
From \eqref{eq:indecomposable-3} and \eqref{eq:indecomposable-4}, $W(0) = 0$ and $W(z_i) = 0$ for $i = 1, \dots, r$. Since $W$ has degree $r$, it must be identically zero. 
The coefficient of $x^{r}$ in $W$ is
\[
Q(h(0)) A r -(m-n) Q(h(0)) A= 0
\]
thus $r = m - n$. Hence $P$ has degree $m - n$. Then, by \eqref{eq:indecomposable}, $Q(x)$ must be constant. This implies that $g$ is linear, which is a contradiction.
\end{proof}

\subsection{Hodge numbers}
\label{sec:hodge}
We end Section \ref{sec:set-up} with a discussion of the Hodge numbers $h^{i,j}$ of the pure Hodge structure $\operatorname{gr}_{\kappa}^W PH^{\kappa}_c(Z_t,\QQ)$. 
We first introduce some notation: the Hodge polynomial of a pure Hodge structure $H$ of weight $n$ is 
\[
\delta^\#(T)\coloneqq \sum_{i=0}^{n} h^{i,n-i}(H) \; T^{i+1}
\] where $h^{i,j}(H)$ are the Hodge numbers of $H$. 

Let $\gamma$ be a prime gamma vector. For each positive integer $N$, define the polynomial
\[ 
\delta_N^{\#}(T):= \sum_{\substack{j=1\\ \gcd(j,N)=1}}^{N-1}  T^{\sum_{i=1}^l \left\{\frac{j \gamma_i }{N} \right\}} 
\] 
where ${x}$ denotes the fractional part of the real number $x$. 
Moreover, for each $N \ge 1$, define the integers:
\begin{align*}
    m_{\pm}(N)&:= \#\{ \gamma_i :  \operatorname{sgn}(\gamma_i)= \pm 1 \text{ and }  N\, | \, \gamma_i \}
\end{align*}
In other words, $m_+(N)$ counts the number of positive $\gamma_i$, which are divisible by the positive integer $N,$ and likewise for $m_-(N)$. For simplicity, we will drop $N$ from the notation and only write $m_\pm$.

\begin{thm}[Rodriguez Villegas]
\label{thm:hodge-deligne}
    Let $\gamma$ be a prime gamma vector and let $(Z, \pi)$ be the pair associated to $\gamma$. Let $Z_t$ be a smooth fiber of $\pi$. The Hodge polynomial $\delta^\#$ of the pure Hodge structure $\operatorname{gr}_{\kappa}^W PH^{\kappa}_c(Z_t)$ is given by
    \[ 
   \delta^\#(T) = \sum_{\substack{N\ge1\\ m_+ > m_-}} \frac{T^{m_+}-T^{m_-}}{T-1} \delta_N^{\#}(T)
    \]
\end{thm}
This theorem appears in \cite{FRV-duke}. A proof was communicated to us by Rodriguez Villegas and will be presented in a forthcoming paper by the same author. Another proof using similar ideas can be found
in \cite{vadym_thesis}.
\smallskip

As a consequence of Theorem \ref{thm:hodge-deligne} we obtain:
\begin{corol} Let $\HH$ be a hypergeometric local system associated to $\gamma$. Then:
\label{cor:rank-eq}
    \begin{equation}
    \label{eq:ranks}
    \rk \HH=\rk \gr_\kappa^W PR^\kappa \pi_{U !}  \CC
    \end{equation}
\end{corol}

\begin{proof}
The right hand side of \eqref{eq:ranks} is given by $\delta^\#(1)$.

Let $\varphi$ denote the Euler totient function and $\Phi_N$ denote the $N$th cyclotomic polynomial. We have
    \begin{align*}
       \delta^\#(1)&= \lim_{t \to 1} \sum_{\substack{N\ge1\\ m_+ > m_-}}  \frac{T^{m_+(N)}-T^{m_-(N)}}{T-1} \delta_N^{\#}(t) \\
        &=\sum_{\substack{N\ge1\\ m_+ > m_-}}  \left( m_{+}(N)-m_-(N) \right) \delta_N^\#(1)\\
        &= \sum_{\substack{N\ge1\\ m_+ > m_-}}  \left( m_{+}(N)-m_-(N) \right) \varphi(N)
    \end{align*}
On the other hand,
\begin{align*}
  Q(T)=  \frac{\prod_{\gamma_i>0} T^{\gamma_i}-1}{\prod_{\gamma_i<0} T^{-\gamma_i}-1}  =  \frac{\prod_{\gamma_i>0}\prod_{N | \gamma_i} \Phi_N(T)}{\prod_{\gamma_i<0}\prod_{N | -\gamma_i} \Phi_N(T)}
    =  \frac{\prod_N \Phi_N(T)^{m_{+}(N)}}{\prod_N \Phi_N(T)^{m_{-}(N)}}
\end{align*}
The rank of $\HH$ is given by the degree of either the numerator or the denominator after cancelling the common factors. The degree of the numerator is exactly
\[\sum_{\substack{N \ge 1\\m_+>m_-}} \left(m_+(N)-m_-(N)\right) \deg (\Phi_N(T)) = \sum_{\substack{N \ge 1\\m_+>m_-}}  \left(m_+(N)-m_-(N)\right) \varphi(N)\]
This proves the claim.
\end{proof}

\section{Gamma vectors with one negative entry} \label{sec:multinomial}
In this section, we give a direct proof of the isomorphism \eqref{eq:H-(Z, pi)} when $\gamma$ is a gamma vector with only one negative entry:
\begin{thm}\label{thm:one-negative-gamma}
Let $\ga$ be a gamma vector with only one negative entry. Let $(Z, \pi)$ be the pair associated to $\gamma$. Then the isomorphism \eqref{eq:H-(Z, pi)} holds.
\end{thm}

We assume that $\ga_1$ is the unique negative entry of $\ga$. 
Let $f = \sum_{j=1}^{l} \gamma_j t^{k_j} x^{m_j}$ define a toric model for $(Z, \pi)$. For a Laurent polynomial $h$, we denote by $c_{\underline{0}}(h)$ the constant term of $h$, i.e., the coefficient of the monomial $1=x^{\underline{0}}$ in $h$.
\smallskip

We first show that a period of $(Z,\pi)$ is annihilated by an irreducible hypergeometric operator whose local system of solution is isomorphic to $\HH$. 

Define the Laurent polynomial
\begin{equation}
g:=(\gamma_1 t^{k_1} x^{m_1})^{-1} \sum_{j=2}^{l} \gamma_j t^{k_j} x^{m_j}
\end{equation}

Let $D^\times$ be a small open punctured disk around $t=0$  and let  $\varepsilon=(\varepsilon_1, \dots, \varepsilon_d)\colon D^\times \to ({\RR_{>0}})^d$  be a continuous function
such that
\[
(l-1) \cdot \max_{j\in \{2, \dots, l\}} \left\{ \frac{\gamma_j}{\gamma_1} |t|^{k_j-k_1} \varepsilon_j^{m_j-m_1} \right\} <1 
\]
over $D^\times$.
Let
\begin{equation}
\delta \coloneqq \{ (x,t) \colon |x_i|=\varepsilon_i(t), \ i=1, \dots, d \} \subset \TT^{d} \times D^\times 
\end{equation} 
Then $|g|<1$ on $\delta$. 

Consider the period of $(Z, \pi)$ defined by
\begin{equation}
    \varphi(t) \coloneqq \frac{1}{(2\pi i)^d} \int_{\delta} \frac{\gamma_1t^{k_1}x^{m_1}}{f} \frac{dx}{x}= \frac{1}{(2\pi i)^d} \int_{\delta} \frac{1}{1+g} \frac{dx}{x}
\end{equation}
where $t \in D^\times$.
By definition of $\delta$, $\varphi$ is well-defined.

Let $\alpha, \beta \in (0,1]^n$ be the unique multisets of rational numbers such that  $\alpha_i - \beta_j \notin \Z$ for all $i, j$ and 
\begin{equation}
\label{eq:gamma-parameters}
\frac{x^{-\ga_1}-1}{\prod_{j=2}^{l}( x^{\gamma_j} -1)} =  \frac{ \prod_{j=1}^{n} (x- e^{2 \pi i \alpha_j} ) }{ \prod_{j=1}^{n} (x- e^{2 \pi i \beta_j} ) }
\end{equation}
Let $H(\alpha, \beta)$ be the corresponding hypergeometric operator. Note $H(\alpha, \beta)$ is irreducible of rank $n$. 
\begin{prop}
The period $\varphi$ is annihilated by $H(\alpha, \beta)$.
\end{prop}
\begin{proof}
Since $|g| <1$ on $\delta$, we can write 
\begin{align*}
  \varphi(t) = \frac{1}{(2\pi i)^d} \int_\delta \sum_{n=0}^\infty (-g)^n \frac{dx}{x}
\end{align*}
By the residue theorem and Lemma \ref{lem:constant-term} below, the last term is equal to 
\[
\sum_{n=0}^\infty c_{\underline{0}}(-g)^n = \sum_{h=0}^\infty (-1)^{ \ga_1 h} \binom{-\ga_1 h} {\ga_2 h,  \dots ,  \ga_l h }(\Gamma t)^{h}
\]
The last term is equal to the hypergeometric function
$F\biggl(\begin{matrix} \alpha_1, \ \ldots, \ \alpha_n \\ \beta_1, \ \ldots, \ \beta_n \end{matrix} \:\bigg|\:\: t \biggr)$ by a standard calculation. The result then follows from Remark \ref{rem:hypergeometric-functions}. 
\end{proof}

\begin{lem}\label{lem:constant-term}
Let $n\ge0$ be an integer. The constant term $c_{\underline{0}}(g^n)$ vanishes unless $n =-\ga_1h$ for some integer $h$. If $n=-\ga_1 h$, then
\[
c_{\underline{0}}(g^n)=\binom{-\ga_1h   } { \ga_2  h,  \dots  ,\ga_l h} (\Gamma t)^{h}
\]
\end{lem}
\begin{proof}

Write $n_1 :=-n$. We have
\begin{align*}
    g^n &= (\gamma_1 t^{k_1} x^{m_1})^{-n}  \left(\sum_{j=2}^{l} \gamma_j t^{k_j} x^{m_j}\right)^n\\
    &=\sum_{\substack{0\le n_2,\dots,n_l\le n \\ n_2+\cdots+n_l=n}}^n \binom{n}{n_2, n_3,  \dots , n_l} \prod_{j=1}^l  (\gamma_j t^{k_j} x^{m_j})^{n_j}\\
    &=\sum_{\substack{0\le n_2,\dots,n_l\le n \\ n_2+\cdots+n_l=n}}^n \binom{n}{n_2\, n_3\,  \dots \, n_l} \left(\prod_{j=1}^l  \gamma_j^{n_j} \right) 
     t^{\sum_{j=1}^l k_j n_j}  x^{\sum_{j=1}^l n_jm_j}
\end{align*}
To get the constant term, we impose $\sum_{j=1}^l n_jm_j=0$.
Note that $\sum_{j=1}^l n_j= 0$. 
Since $\gamma$ spans the lattice of integral relations among $(1,m_1), \dots ,(1,m_l)$ (see Remark \ref{rem:gamma-relations}), there is a nonnegative integer $h$ such that $n_j =  \ga_j h$, $j=1,\dots, l$. Then
$\sum_{j=1}^l k_j n_j =\sum_{j=1}^l k_j \ga_j h=  h$, where the last equality holds by the definition of $(k_1, \dots, k_l)$.
Substituting into the last equation gives the result.
\end{proof}
\smallskip

Note that, by definition of $\alpha$ and $\beta$, $\HH$ is isomorphic to the local system of solutions of the differential equation $H(\alpha, \beta)\cdot \varphi=0$.

We now prove Theorem \ref{thm:one-negative-gamma}.
\begin{proof}[Proof of Theorem \ref{thm:one-negative-gamma}]
Since $\varphi$ is not zero, the integrand $\omega$ is a non-trivial section of $\mathcal{H}^d(\widetilde{Z}_U/U)$.
Let $\Delta$ be the convex hull of the $m_j$, $j=1, \dots, l$. Since $\gamma_1$ is the only negative entry of $\gamma$, $\Delta$ is a simplex, and $m_1$ lies in its interior (see Remark \ref{rem:mj-polytope}). It follows that $\omega$ defines a section of ${W}_{d+1}P\mathcal{H}^d(\widetilde{Z}_U/U)$ (see Theorem \ref{thm:Baty-summary} and Section \ref{subsec:relative-coho}).

Since $H(\alpha, \beta)$  is irreducible, it generates the annihilator $\mathrm{Ann}(\varphi)$ of $\varphi$. Thus $\HH$ is isomorphic to the local system of solutions of $\mathrm{Ann} (\varphi)\cdot g=0$. This is a sub-local system of the local system of solutions of $\mathrm{Ann}(\omega) \cdot g=0$, where $\mathrm{Ann}(\omega)$ denotes the annihilator of $\omega$. Now, since $\omega$ is a section of $\mathcal{W}_{d+1}P\mathcal{H}^d(\widetilde{Z}_U/U)$, this last local system is a quotient of $\gr^W_{d-1}PR^d \widetilde{\pi}_{U !} \CC \simeq \gr^W_{\kappa}PR^{\kappa} {\pi}_{U !} \CC$. By Corollary \ref{cor:rank-eq}, $\HH$ and  $\gr^W_{\kappa}PR^{\kappa} {\pi}_{U !} \CC$ have the same rank, thus they are isomorphic. 
\end{proof}

\begin{rem}
If $\gamma$ has only one positive entry, let $\HH^\prime$ and $(Z', \pi')$ be the hypergeometric local system and the pair associated to $\gamma^\prime=-\gamma$. Since $\gamma'$ has only one negative entry, by Theorem \ref{thm:one-negative-gamma}, the isomorphism \eqref{eq:H-(Z, pi)} holds for $\HH'$ and $(Z', \pi')$. This implies that  \eqref{eq:H-(Z, pi)} also holds for $\HH$ and  $(Z, \pi)$. 
Indeed, letting $\iota \colon t \mapsto 1/t$ be the inversion map, we have, on the one hand, $\HH\simeq \iota^\star \HH^\prime$, on the other hand, 
\[
\gr^W_\kappa R^\kappa \pi_{U \ !} \CC \simeq \iota^\star \gr^W_\kappa R^\kappa \pi'_{U \ !} \CC  
\]
In other words, Theorem \ref{thm:one-negative-gamma} establishes  \eqref{eq:H-(Z, pi)} for all gamma vectors with only one negative or only one positive entry. 
\end{rem}

\begin{rem}
When $\gamma$ has more than one negative entry,  constructing a relative cycle producing a period of $(Z, \pi)$ that is not identically zero seems very hard. 
Note also that, for a vector $\gamma$ whose number of negative entries is greater or equal than that of positive entries, the hypergeoemtric function $F(\alpha, \beta \mid t)$, where $\alpha, \beta \in ((0,1]\cap \QQ)^n$ are the unique multisets satisfying \eqref{eq:gamma-alpha-beta}, is not annihilated by the operator $H(\alpha, \beta)$.
This suggests that $F(\alpha, \beta)$ is not a period of $(Z, \pi)$. 
\end{rem}

\section{Proof of Theorem \ref{thm:non-trivial-isomorphism}} 
\label{sec:proof-thm-1}
Pick a toric model for $(Z,\pi)$ as in Equation \eqref{eq:toric-model}.   
Write 
$
f(x,t)\coloneqq \sum_{j=1}^l \gamma_j t^{k_j} x^{m_j}$ and let $\Delta$ be the convex hull of the $m_j$ as in Section \ref{sec:pair}. 

To prove Theorem \ref{thm:non-trivial-isomorphism}, we first prove a series of preliminary results about the sheaf $\mathcal{H}^d(\widetilde{Z}_U/U)$, where $\widetilde{Z}_U$ is the zero locus of 
$\widetilde{f}=x_0f-1$ in $\TT^{d+1} \times U$, which may be of independent interest.

\begin{rem}
\label{rem:primitive}
Recall that, if $d>1$, the sheaves $W_{d+1} 
\mathcal{H}^d(\widetilde{Z}_U/U)$ and $W_{d+1} 
P\mathcal{H}^d(\widetilde{Z}_U/U)$ are the same, while, if $d=1$, the latter sheaf is the quotient of the first one by the rank-2 constant sheaf $\mathcal{H}^1(\TT^{2} \times U/U)$, see Equation \eqref{eq:Z-Ztilde-diagram}. When $d=1$, all the statements of Sections  \ref{sec:1} and  \ref{sec:2} hold with the same hypotheses and replacing $W_{2}
\mathcal{H}^1(\widetilde{Z}_U/U)$ with $W_{2} 
P\mathcal{H}^1(\widetilde{Z}_U/U)$ in the conclusions. 
\end{rem}

\subsection{Monomial forms and hypergeometric operators.} 
\label{sec:1} 
Pick a nonzero monomial form, i.e., a nonzero element
\begin{equation} \label{eq:mono-form-step1}
\omega= \omega_{(\beta_0, \beta)}=\frac{x^\beta}{f^{\beta_0}} \frac{dx}{x} 
\end{equation}
 of $\mathcal{H}^d(\widetilde{Z}_U/U)$, where 
$\beta\coloneqq (\beta_1, \dots, \beta_d) \in \ZZ^d$, $\beta_0 \in \ZZ_{\geq 1}$, and $(\beta_0, \beta_1, \dots, \beta_d)$ is a lattice point of the cone $C(\Delta)$. 
As in Proposition \ref{prop:u-result}, let $u_1, \dots , u_l$ be coordinates on $(\Gm)^l$, let $u \coloneqq (u_1, \dots, u_l)$, and write $\mathcal{F}\coloneqq \sum_{j=1}^l u_j x^{m_j}$.
The form $\omega$ is the specialisation at $u_j=\gamma_jt^{k_j}$, $i=j, \dots, l$, of the form 
\begin{equation}
\label{eq:mono-form-u}
\Omega= \frac{x^\beta}{\mathcal{F}^{\beta_0}} \frac{dx}{x}   
\end{equation}

By Theorem \ref{theo:form-GKZ}, the periods of $\Omega$ are annihilated by the GKZ hypergeometric differential system of operators:
\begin{equation}\label{eq:form-GKZ-1}
\left\{ 
\begin{array}{l}
\begin{aligned}
& \sum_{j=1}^l \; \theta_j   +\beta_0 \\
& \sum_{j=1}^l m_{ij} \theta_j + \beta_i \quad \text{for} \ i=1, \dots, d\\
& \prod_{\gamma_j>0} \partial_j^{\gamma_j} - \prod_{\gamma_j<0}  \partial_j^{-\gamma_j}  \; 
\end{aligned}
\end{array} \right.
\end{equation}
where we have used that the lattice of integral relations $R$ of Theorem \ref{theo:form-GKZ} is generated by $\gamma$.

The solutions of \eqref{eq:form-GKZ-1} are related to generalised hypergeometric equations as follows:
\begin{prop}\label{prop:reduc-of-GKZ}
Let $g=g(u)$ be an analytic solution to \eqref{eq:form-GKZ-1}. Let  $z\coloneqq u^\gamma$ and  let $\theta \coloneqq  z \frac{d}{dz}$.  Then $g$ has the form $g(u) = u^\eta h_\eta(z),$ where $\eta=(\eta_1, \dots, \eta_l)$
satisfies 
\begin{equation}\label{eq:eta-system}
\left\{ 
\begin{array}{l}
\begin{aligned}
\sum_{j=1}^l  \eta_j  &= -\beta_0\\
\sum_{j=1}^l m_{ij} \eta_j &= -\beta_i \ \text{ for } \  i=1,\dots,d
\end{aligned}
\end{array} \right.
\end{equation}
and $h_\eta(z)$ is annihilated by the hypergeometric differential operator 
\begin{equation}
\label{eq:H-red}
H^{\mathrm{GKZ}}_\eta \coloneqq \Gamma \prod_{\gamma_i >0}  \prod_{j=0}^{\gamma_i-1}\left( \theta + \frac{\eta_i- j}{\gamma_i} \right)- z  \prod_{\gamma_i <0}  \prod_{j=0}^{-\gamma_i-1} \left( \theta+ \frac{\eta_i - j}{\gamma_i} \right)  
\end{equation}
where $\Gamma = \prod_{i=1}^l \gamma_i ^{\gamma_i}$. 
\end{prop}

\begin{proof}  
A computation with power series shows that
an analytic solution $g$ to the first $d+1$ operators in \eqref{eq:form-GKZ-1} is of the form $
g(u) = u^\eta h_\eta(z)$, where $\eta$ and $z$ are as in the statement of the lemma.

Since $u \in (\Gm)^l$, the last operator in \eqref{eq:form-GKZ-1}, applied to $g$, can be rewritten as 
\begin{align*}   
\left[ \prod_{\ga_i >0} \partial_i^{\gamma_i}- \prod_{\gamma_i<0 }  \partial_i^{-\gamma_i}\right] g(u) &=\left[ \prod_{\ga_i >0} u^{\gamma_i} \partial_i^{\gamma_i}- z\prod_{\gamma_i<0 }  u^{-\gamma_i} \partial_i^{-\gamma_i}\right] g(u)  \\
&= \left[  \prod_{\ga_i >0}  \prod_{j=0}^{\gamma_i-1}( \theta_i - j )- z\prod_{\ga_i <0}  \prod_{j=0}^{-\gamma_i-1}( \theta_i - j ) \right] g(u)\\
&= \left[  \prod_{\ga_i >0}  \prod_{j=0}^{\gamma_i-1}( \theta_i - j )- z\prod_{\ga_i <0}  \prod_{j=0}^{-\gamma_i-1}( \theta_i - j ) \right]u^\eta h_\eta(z) \\
&= u^\eta \left[  \prod_{\ga_i >0}  \prod_{j=0}^{\gamma_i-1}( \theta_i +\eta_i- j )- z\prod_{\ga_i <0}  \prod_{j=0}^{-\gamma_i-1}( \theta_i+\eta_i - j ) \right] h_\eta(z)
\end{align*}
Above we have used that $\theta_i$ and $\theta_j$ commute with one another, and the following relations:
\begin{align*}
    u_i^k \partial_i^k &= \theta_i(\theta_i-1)\cdots (\theta_i-k+1) \quad & i=1,\dots,l\\
    \theta_i u^{\eta}&= u^{\eta} (\theta_i+\eta_i) \quad & i=1,\dots,l
\end{align*}

The action of $\theta_i$ on $z$ is given by $\theta_iz=\gamma_i \theta z$. 
Replacing $\theta_i=\gamma_i \theta$, one obtains
\begin{align*}
& u^\eta \left[ \prod_{\ga_i >0}  \prod_{j=0}^{\gamma_i-1}( \gamma_i \theta +\eta_i- j )- z\prod_{\ga_i <0}  \prod_{j=0}^{-\gamma_i-1}( \gamma_i \theta+\eta_i - j) \right] h_\eta(z)\\
& =  u^{\eta} \left[  \prod_{\gamma_i>0}{\gamma_i}^{\gamma_i} \prod_{\ga_i >0}  \prod_{j=0}^{\gamma_i-1}\left( \theta + \frac{\eta_i- j}{\gamma_i} \right)- z \prod_{\gamma_i<0} (\gamma_i) ^{-\gamma_i}  \prod_{\ga_i <0}  \prod_{j=1}^{-\gamma_i} \left( \theta+ \frac{\eta_i - j}{\gamma_i} \right) \right] h_\eta(z)
\end{align*}
It follows that $h_\eta$ is annihilated by $H^{\mathrm{GKZ}}_\eta$.
\end{proof}

\begin{rem}
\label{rem:eta}
The presentation $g(u)=u^\eta h_\eta(z)$, as well as the operator $H^{\mathrm{GKZ}}_\eta$,  depends on the choice of $\eta$. 

Since $\beta/\beta_0$ lies in the polytope $\Delta$, it is the convex linear combination $\beta/\beta_0=\sum_{j=1}^l \mu_j m_j$. 
A solution to \eqref{eq:eta-system} is then given by $\eta_j=-\beta_0 \mu_j$, $j=1, \dots, l$. 
Moreover, if $\overline{\eta}$ is a solution to \eqref{eq:eta-system}, then $\eta=c \cdot \gamma +\overline{\eta}$, with $c \in \CC$.
Finally, there always exists an integer solution to \eqref{eq:eta-system} 
by construction of the vectors $m_j$. 
\end{rem}
 
 \begin{rem}
\label{rem:H-GKZ}
 The operator ${H}^{\mathrm{GKZ}}_\eta$ in \eqref{eq:H-red} is a hypergeometric operator of order equal to $\mathrm{vol}(\gamma)=\sum_{\gamma_i>0} \gamma_i$,  with parameters $\alpha_\eta, \beta_\eta$ given by
 \begin{equation}
     \label{eq:alpha-beta-eta}
 \alpha_\eta \coloneqq \left(\frac{\eta_i -j}{\gamma_i}\right)_{
 \substack{ i \colon \gamma_i <0, \\ j=0, \dots, -\gamma_i-1} }
 \quad \text{and} \quad \beta_\eta \coloneqq 
\left(\frac{\eta_i -j}{\gamma_i}+1 \right)_{ 
\substack{ i \colon \gamma_i >0, \\ j=0, \dots, \gamma_i-1  }} 
\end{equation}
Changing $\eta$ to $\eta^\prime=c \gamma +\eta$ shifts the hypergeometric parameters of the corresponding operator by $c$.
 
The operator ${H}^{\mathrm{GKZ}}_\eta$ is reducible. 
By cancelling pairs $({\alpha_\eta}_i,{\beta_\eta}_j)$ that differ by an integer, one obtains, modulo integers, the parameters of the irreducible hypergeometric operator $H_\gamma$ associated to $\gamma$. 
\end{rem}

Since the matrix ${A}$ in \eqref{eq:matrix-gamma}  has determinant $\pm1$, 
there exists a unique $\bar{\eta} \in \Z^d$ which satisfies the following equations: 
\begin{equation}\label{eq:unique-eta}
\left\{ 
\begin{array}{l}
\begin{aligned}
\sum_{j=1}^l  \bar{\eta}_j  &= -\beta_0\\
\sum_{j=1}^l m_{ij} \bar{\eta}_j &= -\beta_i \ \text{ for } \  i=1,\dots,d\\
\sum_{j=1}^l k_j \bar{\eta}_j &= 0
\end{aligned}
\end{array} \right.
\end{equation}
Write $H^{\text{GKZ}}$ for the hypergeometric differential operator obtained from $H^{\text{GKZ}}_{\bar{\eta}}$ by replacing $z=\Gamma t$.

\begin{lem}
The periods of $\omega$ are annihilated by $H^{GKZ}$.
\end{lem}

\begin{proof}
By Proposition \ref{prop:reduc-of-GKZ}, the periods of $\Omega$ are of the form $u^{\bar{\eta}} h_{\bar{\eta}}(z)$.
Specialization by $u_j=\gamma_j  t^{k_j}$ gives that periods of $\omega$
are of the form
\[
\prod_{j=1}^l \gamma_j^{\bar{\eta}_j } h_{\bar{\eta}}(\Gamma t) 
\] 
(note that it does not depend on  $u$), thus are annihilated by $H^{\mathrm{GKZ}}$. 
\end{proof}

\begin{corol}\label{cor:omega-HGKZ}
The form $\omega$ is annihilated by $H^{\mathrm{GKZ}}$.
\end{corol}

\begin{proof}
A differential operator annihilating all periods of $\omega$ must annihilate $\omega$.
\end{proof}

Changing the monomial form $\omega$ changes the operator $H^{\mathrm{GKZ}}$. However, the operators obtained by cancellation of pairs of integers, as in Remark \ref{rem:H-GKZ}, are all isomorphic to each other. 

In what follows, for a given $\omega$ as in \eqref{eq:mono-form-step1}, we make the dependency on $\omega$ explicit and denote $H^{\mathrm{GKZ}}$ by $H^{\mathrm{GKZ}}_\omega$.

\subsection{Monodromy and genuine singularities at $t=1$.}
\label{sec:2}

Recall that a differential operator $L$ on $\Ps^1$ has a \emph{genuine singularity} at  $p \in \PP^1(\CC)$ if the local system of solutions of the differential equation $L g=0$ has non-trivial local monodromy at $p$.
We call \emph{minimal differential operator} of a form in $\mathcal{H}^{d-1}(Z_U/U)$ or in $\mathcal{H}^{d}(\widetilde{Z}_U/U)$ the unique monic generator of the annihilator of the form in $\CC(t)[\frac{d}{dt}]$. 

We first observe the following fact:
\begin{lem}\label{lem:genuine-sing}
If $R^{d-1}\pi_{U \; !} \QQ$ has non-trivial local monodromy at $t=1$, then there exists a monomial form in $\mathcal{H}^{d}(\tilde{Z}_U/U)$ whose minimal differential operator has a genuine singularity at $t=1$. 
\end{lem}
\begin{proof}
The local system $R^{d-1}{\pi_{U}}_!\QQ$ is dual to the local system $R^{d-1}{\pi_{U}}_* \QQ$ of flat sections of $\mathcal{H}^{d-1}(Z_U/U)$. If  $R^{d-1}{\pi_{U}}_!\QQ$ has non-trivial local monodromy at $t=1$, then there exists a section of $\mathcal{H}^{d-1}(Z_U/U)$ whose minimal differential operator has a genuine singularity at $t=1$.
By the relative Poincar\'e residue mapping, there is a section of $\mathcal{H}^{d}({\tilde{Z}_U/U})$ whose minimal differential operator has a genuine singularity at $t=1$. Moreover, since monomial forms generate the sections of $\mathcal{H}^{d}({\tilde{Z}_U/U})$, the form can be chosen to be monomial.
\end{proof}

Now let $\mathcal{W}_i \coloneqq  {W}_{i} \mathcal{H}^d(\widetilde{Z}_U/U)$ be the $i$th weighted piece. We show that, in fact, if $R^{d-1}\pi_{U \; !} \QQ$ has non-trivial local monodromy at $1$, then the minimal differential operator of some monomial form in the weighted piece $\mathcal{W}_{d+1}$ has a genuine singularity at $1$: 

\begin{prop}\label{prop:coh-form-top} If $R^{d-1}\pi_{U \; !} \QQ$ has non-trivial local monodromy at $t=1$,
then there exists a monomial form in $\mathcal{W}_{d+1}$ whose minimal differential operator has a genuine singularity at $t=1$. 
\end{prop}
\smallskip

To prove Proposition \ref{prop:coh-form-top}, we use the following results on the weighted pieces $\mathcal{W}_i$, whose proofs we defer to the end of this section. 

\begin{lem}\label{lem:weight-submodule}
    For $k=1, \dots, d$, $\mathcal{W}_{d+k}$  is a differential sub-module of  $\mathcal{H}^d(\widetilde{Z}_U/U)$.
\end{lem}

\begin{prop}\label{lem:quotients-Gm}
Suppose that $d \geq 2$. Let $\theta \coloneqq t \frac{d}{dt}$.   Let $\omega$ be a monomial form in $\mathcal{W}_{d+k}$. Then
\[
\theta \omega = c \, \omega \pmod{\mathcal{W}_{d+k-1}} \quad \text{ for some $c \in \QQ$}
\]
Thus
\begin{equation} \label{eq:W-decomposition-t}
     \mathcal{W}_{d+k}/\mathcal{W}_{d+k-1} = \bigoplus_{h} \CC(t)[\theta]/(\theta-c_h)
\end{equation}
for certain numbers $c_h \in \QQ$. 
\end{prop}
\smallskip

\begin{proof}[Proof of Proposition \ref{prop:coh-form-top}]
Let $\omega$ be an output of Lemma \ref{lem:genuine-sing}. Assume that $\omega$ is in $\mathcal{W}_{d+k}\setminus \mathcal{W}_{d+k-1}$, where $k >1$. 
Then, by Proposition \ref{lem:quotients-Gm}, there exists $c$ so that $(\theta-c)  \omega \in \mathcal{W}_{d+i}$ with $d+i < d+k$. 

Let $L$ be the minimal differential operator of $\omega$ and let $r$ be its order. Then one can write:
\[
L \; \omega= \left(\sum_{j=1}^r q_j(t) (\theta-c)^j \right) \omega+q_0(t) \omega
\] for certain rational functions $q_j(t)$ for $j=0, \dots, r$.
Combining this with the fact that the first summand on the right-hand side is in $\mathcal{W}_{d+i}$, one finds that $q_0(t)=0$. In other words, 
$L= L_c  \; (\theta -c)$ for some operator $L_c$ with a genuine singularity at $t=1$. By minimality of $L$, up to dividing by a rational function, the operator $L_c$ is the minimal differential operator of $(\theta -c)\omega \in \mathcal{W}_{d+i}$.
Applying this argument repeatedly, one finds:
\[ 
L\; \omega= L' \prod_j (\theta- c_j) \; \omega
\] for some $c_j \in \QQ$ and some operator $L^\prime$ with a genuine singularity at $t=1$. Moreover, $L^\prime$ is the minimal differential operator of the form $\prod_j (\theta- c) \; \omega \in \mathcal{W}_{d+1}$. This proves the statement.
\end{proof}

\begin{rem} 
\label{rem:orders-ranks}
By Lemma \ref{lem:weight-submodule}, the minimal differential operator of a monomial form in $\mathcal{W}_{d+k}$ has order at most equal to $\rk\mathcal{W}_{d+k}$. In addition, the proof of Proposition \ref{prop:coh-form-top} shows that this order is at most equal to $ \operatorname{rk}  \mathcal{W}_{d+1} + k-1$. 
Note that this does not hold for a general form in $\mathcal{W}_{d+k}$. Indeed, $\mathcal{W}_{d+k}$ must have a cyclic vector, and the rank of $\mathcal{W}_{d+k}$ is generically greater than
$ \operatorname{rk}  \mathcal{W}_{d+1} + k-1$. 
\end{rem}

Let $\mathcal{Z}$ be as in Proposition \ref{prop:u-result}, and let $\mathcal{U}=(\CC^\times)^l \setminus \{\prod_{j=1}^l u_j^{\gamma_j} -\Gamma\}$. Let $\mathscr{W}_i={W}_i\mathcal{H}^d(\widetilde{\mathcal{Z}}_\mathcal{U}/\mathcal{U})$.
Lemma \ref{lem:weight-submodule} follows immediately from the following more general Lemma \ref{lem:weight-submodule-general} as, for any monomial form $\omega$ as in Equation \eqref{eq:mono-form-step1}, we have
\[
\frac{d}{dt} \omega= \sum_{j=1}^l \gamma_j k_jt^{k_j-1} 
\left(\partial_j \Omega\right)_{\big|\substack{u_h=\gamma_h t^{k_h}}} 
\]
In the same way, Proposition \ref{lem:quotients-Gm} follows immediately from the following more general Proposition \ref{lem:quotients-Gm-general} as
\begin{align*}
\theta  \omega   =  \sum_{j=1}^l \gamma_j k_j t^{k_j} \left( {\partial_j \Omega} \right)_{\big|\substack{u_h=\gamma_h t^{k_h}}}=  \sum_{j=1}^l  k_j \left( \theta_j \Omega \right)_{\big|\substack{u_h=\gamma_h t^{k_h}}}
\end{align*}

\begin{lem}\label{lem:weight-submodule-general}
    For $k=1, \dots, d$, $\mathscr{W}_{d+k}$  is a differential sub-module of $\mathcal{H}(\widetilde{\mathcal{Z}}_\mathcal{U}/\mathcal{U})$.
\end{lem}

\begin{prop}\label{lem:quotients-Gm-general}
Suppose that $d \geq 2$.  Let $\Omega$ be a monomial form in $\mathscr{W}_{d+k}$. Then, for $j=1, \dots ,l$,
\[
\theta_j\Omega = d_{j} \, \Omega \pmod{\mathscr{W}_{d+k-1}} \quad \text{ for some $d_{j} \in \QQ$}
\]
Thus
\begin{equation} \label{eq:W-decomposition-u}
     \mathscr{W}_{d+k}/\mathscr{W}_{d+k-1} = \bigoplus_{h} \CC(u_1, \dots, u_l)[\theta_1, \dots, \theta_l]/(\theta_1-d_{h,1}, \dots, \theta_l-d_{h,l})
\end{equation}
for certain numbers $d_{h,j} \in \QQ$.
\end{prop}

\begin{proof}[Proof of Lemma \ref{lem:weight-submodule-general}] If $d=1$, there is nothing to prove. Let $d \geq 2$. 
By Theorem \ref{thm:Baty-summary} it is enough to check that, for all  $i \in \{2, \dots, d\}$, if $\Omega=\Omega_{(\beta_0, \beta)} $ is a monomial form such that
$\beta/\beta_0$ is in the relative interior of a face of $\Delta$ of dimension at least $i$, then $
\partial_j\Omega$ 
is a linear combination of monomial forms $\Omega_{(\beta^j_0, \beta^j)}$ such that $\beta^j/\beta^j_0$ lies in the relative interior of a face of $\Delta$ of dimension at least $i$. 

Fix $i$ and $\Omega$.
We have:
\begin{align*}
\partial_j \Omega =   -\beta_0  \frac{x^{m_{j}+\beta} }{\mathcal{F}^{\beta_0+1}}   \frac{dx}{x}
\end{align*}
Now, we can write
\[
\frac{m_{j}+\beta}{\beta_0+1} =\frac{1}{\beta_0+1} \; m_{j} + \frac{\beta_0}{\beta_0+1}\frac{\beta}{\beta_0}
\]
In other words, $\frac{m_{j}+\beta}{\beta_0+1}$ is a convex linear combination of two points of $\Delta$,  one of which lies in the relative interior of a face of dimension at least $i$. 
Then $\frac{m_{j}+\beta}{\beta_0+1}$ necessarily lies in the relative interior of a face of dimension at least $i$. 
\end{proof}

\begin{proof}[Proof of Proposition \ref{lem:quotients-Gm-general}]
Fix $k \in \{2,\dots, d\}$. It is enough to prove that, for all $j=1, \dots, l$ and for each monomial form $\Omega$ in $\mathscr{W}_{d+k}/\mathscr{W}_{d+k-1}$, there exists $d_j\in \mathbb{Q}$ such that $\theta_j  \Omega=d_j \,  \Omega$. 

Let $\Omega=\Omega_{(\beta_0, \beta)}$ be such a form. 
If $k<d$, then $\beta/\beta_0$ is in the relative interior a face $F$ of $\Delta$ of dimension $d-(k-1)$. If $k=d$, then $\beta/\beta_0$ is either in the relative interior of an edge $F$ of $\Delta$ or is a vertex of $\Delta$. 

For $j=1, \dots, l$, we have 
\begin{equation}
\label{eq:thetaj-omega} 
    \theta_j \Omega = 
    -\beta_0 u_j \frac{x^{m_{j}+\beta}}{\mathcal{F}^{\beta_0+1}} \frac{dx}{x}
    \end{equation}
We have the relation
\begin{equation}
\label{eq:exact-relation-zero}
0= \Omega -\frac{x^\beta \mathcal{F} }{{\mathcal{F}}^{\beta_0+1}}\frac{dx}{x} =
\Omega-\sum_{j=1}^l u_j \frac{x^{m_{j}+\beta} }{{\mathcal{F}}^{\beta_0+1}}\frac{dx}{x} 
\end{equation}   
and, for $i=1,\dots,d$, we have the exactness relations

\begin{equation}
\label{eq:exact-relations}
\begin{split}
    0 = \pm d \left(\frac{x^{\beta}}{\mathcal{F}^{\beta_0}} \frac{dx_1}{x_1} \wedge \dots \wedge \widehat{\frac{dx_i}{x_i}} \wedge \dots \wedge \frac{dx_{d}}{x_d}\right)
    = -\beta_i \Omega + \beta_0 \sum_{j=1}^l u_j m_{ij}\frac{x^{m_{j}+\beta}}{\mathcal{F}^{\beta_0+1}}\frac{dx}{x} 
\end{split} 
\end{equation}
where $d=\sum_{i=1}^d \frac{d}{dx_i} \; dx_i$.

For each $j=1, \dots, l$,
the term 
\[
\frac{x^{m_{j}+\beta} }{{\mathcal{F}}^{\beta_0+1}}\frac{dx}{x}  \in \mathscr{W}_{d+k} 
\]
If it is nonzero modulo $\mathscr{W}_{d+k-1}$, then $\frac{m_{j}+\beta}{\beta_0+1}$  does not belong to the relative interior of any face of $\Delta$ of dimension strictly greater than $d-(k-1)$. 
If $k<d$, this is true if and only if $m_j \in F$.  
If $k=d$ and $\beta/\beta_0$ is in the relative interior of the edge $F$, this is true if and only if $m_j \in F$. If $k=d$ and $\beta/\beta_0$ is a vertex, this is true if and only if $\beta/\beta_0$ and $m_j$ are contained in an edge. In this last case, we denote by $F$ the set of all $m_j$ such that $\beta/\beta_0$ and $m_j$ are contained in an edge.  Note that, in each case, the number of $m_j$ in $F$ is $\leq l-1$.

Then, in $\mathscr{W}_{d+k}/\mathscr{W}_{d+k-1}$, we have that
\begin{equation}
\label{eq:theta-j-cases}
   \theta_j \Omega = \begin{cases}
        -\beta_0 u_j\frac{x^{m_{j}+\beta}}{\mathcal{F}^{\beta_0+1}} & \text{if}   \ \beta/\beta_0 \in F\\
      0 & \text{ otherwise}
  \end{cases} \quad \text{for} \ j=1, \dots, l
\end{equation}
and that
\begin{equation}
\label{eq:relation-face}
    \Omega=\sum_{\substack{m_{j} \in F}}  u_j \frac{x^{m_{j}+\beta} }{{\mathcal{F}}^{\beta_0+1}}\frac{dx}{x} 
\end{equation}
\begin{equation}
    \label{eq:exact-relations-face}
     \frac{\beta_i}{\beta_0} \Omega = \sum_{\substack{m_{j} \in F}}  m_{ij} u_j \frac{x^{m_{j}+\beta}} {\mathcal{F}^{\beta_0+1}}\frac{dx}{x} \quad \text{for } i=1,\dots,d 
\end{equation}
Equations \eqref{eq:relation-face} and \eqref{eq:exact-relations-face}, seen as $d+1$ linear equations in the unknowns $u_j \frac{x^{m_{j}+\beta} }{{F}^{\beta_0+1}}\frac{dx}{x}$, admit the solution 
\begin{equation}
\label{eq:solution}
u_j\frac{x^{m_{j}+\beta} }{{\mathcal{F}}^{\beta_0+1}}\frac{dx}{x}=\mu_j \cdot \Omega
\end{equation}
where $\beta/\beta_0$ is the convex linear combination \[ 
\beta/\beta_0=\sum_{m_j \in F} \mu_j m_j\] (if $\beta/\beta_0=m_h$, then $\mu_h=1$, and $\mu_j=0$ for $j \neq h$). 
This solution is unique since the number of $m_j$ in $F$ is $\leq l-1$ (see the proof of Proposition \ref{prop:u-result}).
Substituting \eqref{eq:solution} in  \eqref{eq:theta-j-cases} gives the statement. 
\end{proof}

\subsection{Proof of Theorem \ref{thm:non-trivial-isomorphism}} \label{sec:3}
Pick a monomial form $\omega \in {W}_{d+1} P\mathcal{H}^d(\widetilde{Z}_U/U)$ such that its minimal differential operator $L\coloneqq L_{\omega}$ has a genuine singularity at $t=1$. Such a form exists by Remark \ref{rem:primitive} and Proposition \ref{prop:coh-form-top}.

By Lemma \ref{lem:weight-submodule}, the Poincar\'e residue mapping, and Corollary \ref{cor:rank-eq}, we have that

\begin{equation} 
\label{eq:inequality-1}
\mathrm{ord} \, L \leq \rk \mathcal{W}_{d+1}P\mathcal{H}^d(\widetilde{Z}_U/U)= \rk {W}_{d-1}(P\mathcal{H}^{d-1}(Z_U/U))= \rk\HH
\end{equation}

Let $H \coloneqq H^{\mathrm{GKZ}}_\omega$.
For $D=L$ or $D$, let $\rho_{D}$ be the monodromy representation of the local system of solutions of $D \cdot \varphi=0$, and let $\rho_{D,s}$ be the local monodromy at $t=s$. 

By Corollary \ref{cor:omega-HGKZ},  $\omega$ is annihilated by $H$. Thus  $H$ is a left multiple of $L$. Hence, $\rho_L$ is a sub-representation of  $\rho_H$. 

Since $H$ is hypergeometric, $\rho_{H, 1}$ is a pseudoreflection (see Section \ref{sec:hgm-ls}). Therefore, $\rho_{L,1}$ is either a pseudoreflection or is the identity. 
Since $L$ has a genuine singularity at $t=1$, the latter cannot happen. 
It follows that the quotient representation 
\[ \bar{\rho}=\rho_H/\rho_L
\] is trivial at $t=1$, that is,  $\bar{\rho}_{\infty}= {\bar{\rho}_{0}}^{\; -1}$. In particular, the eigenvalues of 
$\bar{\rho}_{\infty}$ and of ${\bar{\rho}_{0}}^{\; -1}$ coincide.
The eigenvalues of 
$\bar{\rho}_{\infty}$ and of ${\bar{\rho}_{0}}^{\; -1}$ form a subset of the eigenvalues of 
${\rho}_{H,\infty}$ and of ${{\rho}_{H,0}}^{-1}$, which are given, respectively, by the $\alpha_{\bar{\eta},i}$ and the $\beta_{\bar{\eta},i}$ in \eqref{eq:alpha-beta-eta} with $\bar{\eta}$ the unique solution to \eqref{eq:unique-eta}. Therefore, the dimension of $\bar{\rho}$ is bounded above by the number of $\alpha_{\bar{\eta},i}$ and $\beta_{\bar{\eta},i}$ that differ by an integer, i.e., by the difference $\mathrm{ord}\,H - \rk \HH$ (see the end of Remark \ref{rem:H-GKZ}). 
This implies that 
\begin{equation}
\label{eq:inequality-2}
\mathrm{ord} \,L \geq \rk \HH 
\end{equation} 

Combining \eqref{eq:inequality-1} and \eqref{eq:inequality-2} we find
$
\mathrm{ord} L= \rk\HH
$. It follows that the local system of solutions of $L \cdot \varphi =0$ is isomorphic to 
the local system
$\gr^W_{\kappa} PR^{\kappa} \pi_{U !} \CC$. 
Moreover, the local monodromies $\rho_{L,s}$ have the same eigenvalues as the local monodromies $\rho_s$ of $\HH$.
But then, by Proposition \ref{pro:levelt}, $\gr^W_{\kappa} PR^{\kappa} \pi_{U !} \CC$ is isomorphic to $\HH$.  \qed
\smallskip

\section{Proof of Theorem \ref{thm:curves-even}}
\label{sec:proof-thm-2}
We begin by stating two results that form the key ingredients of the proof.
As before, we write $Z_t=\pi^{-1}(t)$.
We denote by $x$ the unique singular point of $Z_1$, and by $F_x$ the Milnor fibre at $x$. 
We refer to \cite{Milnor} for the notions of Milnor fibration and Milnor fibre.
We write $\mathrm{id}$ for the identity map.

\begin{prop} 
\label{pro:pro-sequence} 
Let $t$ be a regular value of $\pi$. Let $\rho_{\pi,1}$ be the local monodromy of $R^{d-1} {\pi_U}_{ !}  \QQ$ at $1$. Then:
\begin{enumerate}
\item[(i)] There is an exact sequence: 
\begin{equation}
\label{eq:sequence}
0 \to H^{d-1}_c(Z_1, \QQ) \to H^{d-1}_c(Z_t, \QQ) \to H^{d-1}(F_x, \QQ) \to H^d_c(Z_1, \QQ) \to H^d_c(Z_t, \QQ) \to 0
\end{equation} and  $H^{d-1}(F_x, \QQ)\simeq \QQ$.
\item[(ii)] 
The monodromy operators $M_1$ on the terms of \eqref{eq:sequence} given by
\begin{itemize}
\item for $i=d-1,d$, $M_1\coloneqq \mathrm{id}$ on $H^{i}_c(Z_1, \QQ)$
\item for $i=d-1,d$, $M_1\coloneqq \rho_{\pi, 1}$ on $H^{i}_c(Z_t, \QQ)$
\item $M_1\coloneqq-\mathrm{id}$ on $H^{d-1}(F_x, \QQ)$ if $d$ is odd, $M_1\coloneqq\mathrm{id}$ on $H^{d-1}(F_x, \QQ)$ if $d$ is even
\end{itemize}
are compatible with  \eqref{eq:sequence}.
\end{enumerate}
\end{prop}

\begin{cor}
\label{cor:if-iso}
If the map 
\begin{equation}
\label{eq:sp_d}
 	H^d_c(Z_1, \QQ) \to H^d_c(Z_t,  \QQ)
 \end{equation}
in the sequence \eqref{eq:sequence} is an isomorphism, then $\rho_{\pi,1}$ is non-trivial.  
\end{cor}

\begin{rem}
\label{eq:can}
By the first part of Proposition \ref{pro:pro-sequence} , the map \eqref{eq:sp_d} is an isomorphism if and only if the map 
$H^{d-1}_c(Z_t,  \QQ) \to H^{d-1}(F_x, \QQ)$ in \eqref{eq:sequence} is not zero. 
Moreover, 
if $\rho_{\pi,1}$ is non-trivial, the latter map is not zero. Indeed, if the map was zero, then, by the exactness of \eqref{eq:sequence}, the map $H^{d-1}_c(Z_1,\QQ) \to H^{d-1}_c(Z_t,\QQ)$ in \eqref{eq:sequence} would be an isomorphism. But then, since this map is compatible with the operators $M_1$, and $M_1$ is the identity on $H^{d-1}_c(Z_1,\QQ)$ and $\rho_{\pi,1}$ on $H^{d-1}_c(Z_t,  \QQ)$, we would obtain that $\rho_{\pi,1}=\mathrm{id}$.    
It follows that the condition of the corollary is equivalent to the non-triviality of $\rho_{\pi,1}$, although we will not use this fact to prove the theorem. 
 \end{rem}
\smallskip

We defer the proofs of Proposition \ref{pro:pro-sequence} and Corollary
\ref{cor:if-iso} to the end of the section. We now prove Theorem \ref{thm:curves-even}. 

\begin{proof}[Proof of Theorem \ref{thm:curves-even}]
By the exactness of the sequence \eqref{eq:sequence}, the map \eqref{eq:sp_d} is surjective. 
We prove that, under assumption (1) or assumption (2), this map is an isomorphism. By Corollary \ref{cor:if-iso}, this ends the proof.

(1) We have that $d=2$. 
We argue that both the domain and the target of the map \eqref{eq:sp_d} have rank one.
Recall that for a curve $X$, the rank of $H^2_c(X, \QQ)$ is the number of its irreducible components.
By Proposition \ref{prop:irreduciblity}, the groups $H^{2}_c(Z_1,\mathbb{Q})$ and $H^{2}_c(Z_t,\mathbb{Q})$  have rank $1$. By surjectivity of \eqref{eq:sp_d}, it follows that \eqref{eq:sp_d} is an isomorphism.

(2) 
We have that $d$ is odd. The map $H^{d-1}(F_x, \QQ) \to H^d_c(Z_1,\QQ)$ in \eqref{eq:sequence} is the zero map.  
Indeed, this map is compatible with the operators $M_1$, and $M_1$ acts non-trivially on its one-dimensional domain and trivially on the target.
By exactness of \eqref{eq:sequence}, the map \eqref{eq:sp_d} is an isomorphism.
\end{proof}

\begin{rem}\label{rem:irreducibility-is-not-enough}
If the varieties $Z_t$ have odd dimension $d-1>1$, then irreducibility alone does not  guarantee that $H^d_c(Z_1, \QQ)$ and $H^d_c(Z_1, \QQ)$ have the same rank (see for instance \cite{Lin20}). 
\end{rem}
\smallskip

We now prove Proposition \ref{pro:pro-sequence} and Corollary \ref{cor:if-iso}. The proofs rely on the formalism of nearby and vanishing cycles. References for this material are \cite{DeligneKatz, Dimca, Sawin, ML, Massey}.

\begin{proof}[Proof of Proposition \ref{pro:pro-sequence}]
Let $\pi \colon Z \to \CC^\times$ be as above. 
Let $D \subset \CC^\times$ be a small analytic neighbourhood of $1 \in \CC^\times$. By abusing notation, we denote by $\pi \colon Z \to D$ the restriction $\pi \colon \pi^{-1}(D) \to D$. 

(i) Since $\pi\colon Z \to D$ is not proper, we first show that it is cohomologically tame with respect to a suitable compactification, in the sense of \cite[Definition 6.2.12]{Dimca}. This property will allow us to use the nearby and vanishing cycle functors of the compactification to derive the sequence \eqref{eq:sequence}.

Let $\bar{\pi} \colon V \to D$ be a fibrewise compactification as in Theorem \ref{thm:compactifying-delta}. Then, the boundary $B\coloneqq V \setminus Z$ is a relative normal crossing divisor. 
We have the following decomposition of $B$ in terms of intersections of its components.
The divisor $B$ is the union of smooth irreducible divisors $B_1, \dots, B_n$. For each subset $I$ of $\{1, \dots, n\}$ define 
\[
B^I \coloneqq \bigcap_{i \in I} B_i
\] Since the intersections of the $B_i$ are transverse, $B^I$ is smooth. 
For each $k \geq 1$, let \[B^k \coloneqq \bigcup_{|I|=k} B^I \setminus \left( \cup_{|I|=k+1} B^I\right) \]
Then $B^k$ is the disjoint union of smooth varieties 
$\{B^k_i\}_{i=1, \dots, n_k}$ which are locally closed in $B$,
and $B$ is the disjoint union of all such varieties
$\{B^k_i\}_{k \geq 1, i=1, \dots, n_k}$.
We obtain a decomposition 
\[ V= Z \cup \bigcup_{k,i} B^k_i
\]
such that, for each $k$ and $i$, the restriction $\bar{\pi}^k_i \coloneqq \bar{\pi}_{|B^k_i}$ is smooth. 
In other words, $\pi$ is cohomologically tame with respect to the $\bar{\pi}$. 

Denote by $Z_s$, $V_s$ the fibre at $s \in D$  of $\pi$, $\bar{\pi}$. Let $t \in D \setminus \{1\}$. Consider the diagram: 
\begin{equation}
 \label{eq:diagram}
  \begin{tikzcd}
Z_1 \rar{g_1}  \dar{j_1} &   Z  \dar{j} & \lar{g_t}   Z_t \dar{j_t}\\
V_1 \rar{\bar{g}_1}  \dar{} &   V \dar{\bar{\pi}}   & \lar{\bar{g}_t}  V_t  \dar{}\\
\{ 1 \}  \rar{} & D & \lar{} \{ t \} 
\end{tikzcd}
\end{equation}
For $f=\pi,\bar{\pi}$, denote by $\psi_{f}$ the nearby cycle functor of $f$ and by $\varphi_{f}$   the vanishing cycle functor of $f$. 
For any complex $\mathcal{F}$ of constructible sheaves on $V$ there is a distinguished triangle on $V_1$:
\begin{equation}
\label{eq:triangle}
 {\bar{g}_1}^\star \mathcal{F} \xrightarrow{sp} \psi_{\bar{\pi}} \mathcal{F} \xrightarrow{can} \varphi_{\bar{\pi}}\mathcal{F}  \xrightarrow{+1}
\end{equation}
where ${sp}$ and ${can}$ denote the specialization morphism and the canonical morphism.
Taking hypercohomology gives the long exact sequence
\begin{equation}
\label{eq:hypercohomology}
    \dots \to \HH^i(V_1, {\bar{g}_1}^\star \mathcal{F}) \to  \HH^i(V_1, \psi_{\bar{\pi}} \mathcal{F}) \to \HH^i(V_1, \varphi_{\bar{\pi}} \mathcal{F}) \to  \HH^{i+1}(V_1, {\bar{g}_1}^\star \mathcal{F}) \to \dots
\end{equation}

For a variety $X$, denote by $\QQ_X$ the constant sheaf on $X$.
Setting $\mathcal{F}= Rj_{!} \QQ_Z$, we have that:
\begin{equation}
\label{eq:pullback}
\HH^i(V_1, {\bar{g}_1}^\star \mathcal{F})= 
\HH^i(V_1,  R_{j_1!}\QQ_{Z_1}) =H^i_c(Z_1, \QQ)
\end{equation}
and 
\begin{equation}
\label{eq:nearby}
\HH^i(V_1, {\psi_{\bar{\pi}} }\mathcal{F})= \HH^i(V_t, {\bar{g}_t}^\star \mathcal{F})=\HH^i(V_t, R_{j_t !} \QQ_{Z_t})=H^i_c(Z_t, \QQ)
\end{equation}
where the first equality in \eqref{eq:pullback} and the second equality in \eqref{eq:nearby} hold by proper base change, and the first equality in \eqref{eq:nearby} follows from \cite[Theorem 4.14]{ML}, since $\bar{\pi}$ is proper. 
Finally, 
\begin{equation}
\label{eq:vanishing}
\begin{aligned}
\HH^i(V_1, {\varphi_{\bar{\pi}} }\mathcal{F})&=\HH^i(V_1, R_{j_1 !} {\varphi_{{\pi}} }\QQ_{Z})= \HH^i(V_1, H^{d-1}(F_x, \QQ)[-(d-1)])\\
& = \left\{ \begin{array}{ll}
H^{d-1}(F_x, \QQ) & \textrm{if} \ i=d-1\\
0 & \textrm{otherwise}
\end{array} \right.
\end{aligned}
\end{equation}
where the first equality holds by cohomological tameness of $\pi$  with respect to $\bar{\pi}$ (see the proof of \cite[Lemma 6.2.14]{Dimca}), and the second equality holds as $x$ is the unique singular point of $Z_1$. 

Then \eqref{eq:hypercohomology} yields  the isomorphisms:
\begin{equation} 
\label{eq:isomorps} 
H^i(Z_1, \QQ) \xrightarrow{\simeq} H^i(Z_t, \QQ) \quad \text{for} \ i \neq d-1,d
\end{equation}
and the exact sequence \eqref{eq:sequence}.
Moreover, since $x \in Z_1$ is an ordinary double point, one has that
$H^{d-1}(F_x, \QQ)\simeq \QQ$.

(ii) The functors $\psi_{\bar{\pi}}$ and $\varphi_{\bar{\pi}}$  
are endowed with monodromy transformations which induce the trivial action on
$H^i_c(Z_1, \QQ)$, the action of $\rho_{\pi,1}$ on $H^i_c(Z_t, \QQ)$, and the Milnor monodromy action on $\widetilde{H}^{d-1}(F_x, \QQ)$. Moreover, \eqref{eq:isomorps} and \eqref{eq:sequence} are compatible with these actions (see \cite[Section 4.2]{Dimca}). Since $x \in Z_1$ is an ordinary double point, the Milnor monodromy acts as $-\mathrm{id}$ if  $Z_1$ is even-dimensional, as $\mathrm{id}$ if $Z_1$ is odd-dimensional (see \cite[Section 3.3]{Dimca}). 

\end{proof}

\begin{rem}
\label{rem:more-ODPs} If $Z_1$ had more than a singular point, but only isolated singularities, the same argument would yield the sequence:
\[ 
0 \to H^{d-1}_c(Z_1, \QQ) \to H^{d-1}_c(Z_t, \QQ) \to \bigoplus_{x \in \mathrm{S}(Z_1) } H^{d-1}(F_x, \QQ) \to H^d_c(Z_1, \QQ) \to H^d_c(Z_t, \QQ) \to 0
\] where $\mathrm{S}(Z_1)$ is the set of singular points of $Z_1$. 
\end{rem}

\begin{proof}[Proof of Corollary \ref{cor:if-iso}]
We continue to use the notation of the previous proof. 
Let $i \colon B \to V$ be the inclusion. There is a distinguished triangle:
\begin{equation}
\label{eq:B-triangle}
Rj_{!}  j^\star \QQ_V = Rj_{!}  \QQ_Z \xrightarrow{} \QQ_V \xrightarrow{} i_\star i^\star \QQ_V= i_\star \QQ_B \xrightarrow{+1}	
\end{equation}
Combining it with the triangle \eqref{eq:triangle}, we obtain the diagram: 
\begin{equation}
    \label{eq:diagram-triangles}
\begin{tikzcd}
\textcolor{white}{---} & \textcolor{white}{---} & \textcolor{white}{---}& \textcolor{white}{---}\\
{\bar{g}_1}^\star i_\star \QQ_B \rar{}  \uar{+1}  &    \psi_{\bar{\pi}} i_\star \QQ_B  \rar{} \uar{+1} & \varphi_{\bar{\pi}}  i_\star \QQ_B \rar{+1}  \uar{+1} & \textcolor{white}{---}\\
{\bar{g}_1}^\star \QQ_V  \rar{}  \uar{} &   \psi_{\bar{\pi}} \QQ_V \rar{} \uar{} & \varphi_{\bar{\pi}} \QQ_V \rar{+1}  \uar{}  & \textcolor{white}{---} \\
{\bar{g}_1}^\star Rj_{!}  \QQ_Z  \rar{}   \uar{}& \psi_{\bar{\pi}} Rj_{!}  \QQ_Z  \rar{} \uar{ }& \varphi_{\bar{\pi}} Rj_{!}  \QQ_Z \rar{+1} \uar{} & \textcolor{white}{---} \\
\end{tikzcd}
\end{equation}
Since $\pi$ is cohomologically tame with respect to $\bar{\pi}$, we have that the morphism $\varphi_{\bar{\pi}} Rj_{!}\QQ_Z  \to \varphi_{\bar{\pi}} \QQ_V$ in the last column is an isomorphism 
\cite[Theorem 6.2.15]{Dimca}. It follows that 
$\varphi_{\bar{\pi}}  i_\star \QQ_B =0$. 

Let $\bar{\pi}_B\coloneqq \bar{\pi}_{|B}$ be the restriction of $\bar{\pi}$ to $B$. Write $B_t \coloneqq {\bar{\pi}_B}^{-1}(t)$ and denote by $i_t \colon B_t \to V_t$ the inclusion. 
Applying hypercohomology to the diagram above, we obtain the following commutative diagram:
\begin{equation}
\label{eq:diagram-Bt-Zt}
\begin{tikzcd}
\textcolor{white}{---} & \textcolor{white}{---} & \textcolor{white}{---} & \textcolor{white}{---}\\
0 \rar & H^{d-1}(B_1, \QQ) \rar{\simeq}  \uar{+1}  &    H^{d-1}(B_t,\QQ)  \rar{} \uar{+1} & 0 \rar{+1}  \uar{+1} & \textcolor{white}{---}\\
0 \rar & H^{d-1}(V_1, \QQ) \rar{}  \uar{} &    H^{d-1}(V_t, \QQ)  \rar{can} \uar{{i_{t}}^{\star}} & \QQ \rar{+1} \uar{} & \textcolor{white}{---} \\
0 \rar & H_c^{d-1}(Z_1, \QQ) \rar{}   \uar{} &    H_c^{d-1}(Z_t, \QQ)  \rar{can} \uar{{j_t}_\star} & \QQ \rar{+1} \uar{\simeq} & \textcolor{white}{---} \\
\textcolor{white}{---}   & \textcolor{white}{---} \uar{} & \textcolor{white}{---} \uar{} &  \textcolor{white}{---} \uar{} & \textcolor{white}{---} 
\end{tikzcd}
\end{equation}
where the two maps denoted by $can$ are induced by the canonical morphism in \eqref{eq:triangle}. 

As before, each term of the sequences carries a monodromy action, and the diagram is compatible with these actions. 
As in the statement of Proposition \ref{pro:pro-sequence}, 
we denote all the monodromy operators uniformly by $M_1$. 
For $x \in H^{d-1}(V_t, \QQ)$ we have that $can(x)=\pm \langle x, \delta \rangle$,  where the sign $\pm$ is determined by dimension of $V_t$,  $\delta \in H^{d-1}(V_t, \QQ)$ is the class of the vanishing cycle, and $\langle \cdot, \cdot \rangle$ is the intersection pairing, 
and the action of $M_1$ on $H^{d-1}(V_t, \QQ)$ is given by the Picard--Lefschetz formula
\begin{equation} 
\label{eq:PL} M_1(x) = x +can(x) \delta 
\end{equation} see \cite[Section 3.3]{Dimca92}.

The sequence in the bottom row of \eqref{eq:diagram-Bt-Zt} is \eqref{eq:sequence}. Assume that the map \eqref{eq:sp_d} is an isomorphism. Then
the map ${can} \colon H^{d-1}_c(Z_t, \QQ) \to \QQ $ is not zero. Let $\beta \in H^{d-1}_c(Z_t, \QQ)$ be such that $\mathrm{can}(\beta) \neq 0$ and let $\alpha={j_{t}}_{\star} (\beta)$.  By the commutativity of the diagram, we have that $\mathrm{can}(\alpha) \neq 0$ (in particular $\alpha \neq 0$ and $\delta \neq 0$). Moreover, by exactness of the sequence in the central column of \eqref{eq:diagram-Bt-Zt}, we have that ${i_{t}}^{\star} (\alpha)=0$.
Then: 
\[ 0= M_1({i_{t}}^{\star}  (\alpha))= {i_{t}}^{\star} (M_1(\alpha))={i_{t}}^{\star}  \left(\alpha + \mathrm{can}(\alpha) \delta\right)
=\mathrm{can}(\alpha) {i_{t}}^{\star} (\delta) 
\]
where the second equality holds by the compatibility of the diagram with monodromy and the third equality holds by \eqref{eq:PL}. It follows that ${i_{t}}^{\star}(\delta)=0$. This implies that there this a non-trivial class $\epsilon \in H^{d-1}_c(Z_t, \QQ)$ such that $j_{t \star}(\epsilon)=\delta$. Now
\[ {j_{t}}_{\star}(M_1(\beta))=M_1(\alpha)=\alpha + \mathrm{can}(\alpha)\delta= {j_{t}}_{\star}(\beta+\mathrm{can}(\alpha)\epsilon) \] 
Hence
\begin{equation}
\label{eq:Z-PL}
M_1(\beta)=\beta + \mathrm{can}(\beta) \epsilon +\zeta
\end{equation} for some $\zeta$ in the kernel of ${j_{t}}_{\star}$. It follows that $M_1$ acts non-trivially on $H^{d-1}_c(Z_t, \QQ)$. This ends the proof.
\end{proof}

We end the section with a few remarks.

\begin{rem}
\label{rem:even-proof}	
If $Z_t$ is even-dimensional, Corollary \ref{cor:if-iso} has an easier proof. Indeed,
proving the corollary is the same as showing that, if $\rho_{\pi,1}$ is trivial, then $can \colon H^{d-1}_c(Z_t, \QQ) \to \QQ$ is zero. But this follows immediately from the second part of Proposition \ref{pro:pro-sequence}. Indeed, $can$ is compatible with $M_1$, and $M_1=\rho_{\pi,1}$ on  $H^{d-1}_c(Z_t, \QQ)$,  $M_1 =- \mathrm{id}$ on $\QQ$.
\end{rem}

\begin{rem}
\label{rem:consequences}	
The first part of the proof of the corollary shows that $M_1$ acts trivially on $H^{l}(B_t, \QQ)$ for any $l$ (without any assumption on the map \eqref{eq:sp_d}). 
The second part of the proof shows that, for $\beta \in H^{d-1}_c(Z_t, \QQ)$, 
\[can(\beta)\neq 0 \implies M_1(\beta) \neq \beta\]
which is a stronger statement than the one in the corollary. 
In particular this yields an invariant-cycle theorem for $Z_t$, i.e., the image of  $H^{d-1}_c(Z_1, \QQ)$ is the $\rho_{\pi,1}$-invariant subspace of
$H^{d-1}_c(Z_t, \QQ)$.
\end{rem}

\begin{rem}
\label{rem:more-generality}
    It is clear from the our proofs that Proposition \ref{pro:pro-sequence} and Corollary \ref{cor:if-iso} hold more generally for morphisms $\pi \colon Z \to \CC^\times$ with isolated critical values and such that the fibre over one such critical value $s$ has a unique ordinary double point and the restriction of $\pi$ over a small analytic neighbourhood of $s$ is cohomologically tame. 
    The same is true for Theorem \ref{thm:curves-even} (assuming irreducibility of fibres in case (1)). 
\end{rem}

\section*{Appendix.  Computer algebra in an example}

Throughout the paper, we have considered the example of the gamma vector $\gamma=(-5,-2,3,4)$. In Example \ref{exa:curve-cone}, we provided a basis $\omega_1, \omega_2, \omega_3, \omega_4$ for the weighted piece $W_3\mathcal{H}^2(\widetilde{Z}_U/U)$. 

Using computer algebra, for each class $\omega_i$, we computed the minimal differential operator that annihilates it. This operator is, as expected, an irreducible hypergeometric differential operator with parameters $\alpha, \beta$ corresponding to $\gamma$.
Below, we collect the four classes and the corresponding operators.

We note that the parameters associated with any of the two classes are the same modulo $\Z$ but are not identical.

We also note that, for each $\alpha,\beta$ in Table \ref{tab:computer-algebra},  the hypergeometric function 
\[
F\biggl(\begin{matrix} \alpha_1, \ \ldots, \ \alpha_4 \\ \beta_1, \ \ldots, \ \beta_4 \end{matrix} \:\bigg|\:\: t \biggr)
\] is not annihilated by the hypergeometric operator $H(\alpha, \beta)$.

\begin{table}[ht!]
\centering
\renewcommand{\arraystretch}{2.4} 
\setlength{\tabcolsep}{12pt}
\begin{tabular}{|c|c|c|}
\hline
\text{Form} & $\alpha$ & $\beta$ \\
\hline
$w_1=\dfrac{x_1 x_2}{f}\dfrac{dx}{x}$
&
$\left( \dfrac{2}{5}, \dfrac{3}{5}, \dfrac{4}{5}, \dfrac{6}{5} \right)$
&
$\left( \dfrac{2}{3}, \dfrac{3}{4}, \dfrac{5}{4}, \dfrac{4}{3} \right)$
\\
\hline

$w_2=\dfrac{x_1 x_2^2}{f}\dfrac{dx}{x}$
&
$\left( \dfrac{4}{5}, \dfrac{6}{5}, \dfrac{7}{5}, \dfrac{8}{5} \right)$
&
$\left( \dfrac{5}{4}, \dfrac{4}{3}, \dfrac{5}{3}, \dfrac{7}{4} \right)$
\\
\hline

$w_3=\dfrac{x_1^2 x_2^2}{f^2}\dfrac{dx}{x}$
&
$\left( \dfrac{4}{5}, \dfrac{6}{5}, \dfrac{7}{5}, \dfrac{8}{5} \right)$
&
$\left( \dfrac{3}{4}, \dfrac{5}{4}, \dfrac{4}{3}, \dfrac{5}{3} \right)$
\\
\hline

$w_4=\dfrac{x_1^2 x_2^4}{f}\dfrac{dx}{x}$
&
$\left( \dfrac{8}{5}, \dfrac{9}{5}, \dfrac{11}{5}, \dfrac{12}{5} \right)$
&
$\left( \dfrac{5}{3}, \dfrac{7}{4}, \dfrac{9}{4}, \dfrac{7}{3} \right)$
\\
\hline
\end{tabular}
\caption{\label{tab:computer-algebra} The four generators $w_i$ of $W_3\mathcal{H}^2(\widetilde{Z}_U/U)$ and the parameters $\alpha,\beta$ of the corresponding irreducible hypergeometric operators.}
\end{table}

\bibliographystyle{amsalpha}
\bibliography{bib-aagg}
\end{document}